\documentclass[12pt]{amsart}

\usepackage{float}
\usepackage{amsmath}
\usepackage{amsfonts}
\usepackage{amssymb}
\usepackage{graphicx}
\usepackage{hyperref}
\usepackage{mathrsfs}
\usepackage{pb-diagram}
\usepackage{epstopdf}
\usepackage{amsmath,amsfonts,amsthm,enumerate,amscd,latexsym,curves}
\usepackage{bbm}
\usepackage{mathrsfs}
\usepackage{epsfig,epsf,xypic,epic}
\usepackage{caption}
\usepackage{subcaption}
\usepackage{tikz}
\usepackage{geometry}
\usepackage{tikz-cd}
\usepackage{enumitem}

\usetikzlibrary{matrix,arrows,decorations.pathmorphing}

\newtheorem{theorem}{Theorem}[section]

\newtheorem{lemma}[theorem]{Lemma}
\newtheorem{corollary}[theorem]{Corollary}
\newtheorem{question}[theorem]{Question}

\newtheorem{definition}[theorem]{Definition}

\newtheorem{example}[theorem]{Example}

\newtheorem{proposition}[theorem]{Proposition}

\newtheorem{remark}[theorem]{Remark}

\newtheorem{construction}[theorem]{Construction}

\newcommand{\Int}{\mathrm{Int}}

\newcommand{\hra}{\hookrightarrow}

\newcommand{\inner}[1]{\langle #1\rangle}

\newcommand{\bb}[1]{\mathbb{#1}}
\newcommand{\cu}[1]{\mathcal{#1}}
\newcommand{\til}[1]{\widetilde{#1}}

\newcommand{\ol}[1]{\overline{#1}}

\newcommand{\ul}[1]{\underline{#1}}
\newcommand{\Spec}[1]{\text{Spec}({#1})}

\def\codim{\text{codim}}

\begin{document}
	
	\title[]{Tropical Lagrangian multi-sections and toric vector bundles}
	
	\author[Suen]{Yat-Hin Suen}
	\address{Center for Geometry and Physics\\ Institute for Basic Science (IBS)\\ Pohang 37673\\ Republic of Korea}
    \curraddr{School of Mathematics\\Korea Institute for Advanced Study\\ Seoul \\Republic of Korea}
	\email{yhsuen@kias.re.kr}
	\date{\today}

	\begin{abstract}
	We introduce the notion of tropical Lagrangian multi-sections over a fan and study its relation with toric vector bundles. We also introduce a ``SYZ-type" construction for toric vector bundles which gives a reinterpretation of Kaneyama's linear algebra data. In dimension 2, this ``mirror-symmetric" approach provides us a pure combinatorial condition for checking which rank 2 tropical Lagrangian multi-sections arise from toric vector bundles.
	\end{abstract}
	
	\maketitle
	
	
	\section{Introduction}
	
	Toric geometry is an interaction between algebraic geometry and combinatorics. Difficult problems in algebraic geometry can usually be simplified in the toric world. Toric geometry also plays a key role in the current development of mirror symmetry. It provides a huge source of computable examples for mathematicians and physicists to understand mirror symmetry \cite{Abouzaid06,Abouzaid09,Chan09,Chan-Leung10a,Chan-Leung12,CPU14,Fang08,FLTZ11b,FLTZ12}. The famous Gross-Siebert program \cite{GS1,GS2,GS11} applies toric degenerations to solve the reconstruction problem in mirror symmetry, which is often referred to as the algebro-geometric SYZ program \cite{SYZ}.
	
	In this paper, we study the combinatorics of toric vector bundles. The study of toric vector bundles can be dated back to Kaneyama's classification \cite{Kaneyama_classification} using linear algebra data and also Klyachko's classification \cite{Klyachko1} using filtrations indexed by rays in the fan. Payne \cite{branched_cover_fan,moduli_toric_bundle} studied toric vector bundles and their moduli in terms of piecewise linear functions defined on cone complexes. Motivated by the work of Payne, the notion of tropical Lagrangian multi-sections was first introduced by the author of this paper in \cite{Suen_TP2} and generalized to arbitrary 2-dimensional integral affine manifolds with singularities in a joint work with Chan and Ma \cite{CMS_k3bundle}. 
	
	We begin by recalling some elementary facts about toric varieties and toric vector bundles in Section \ref{sec:toric}. In Section \ref{sec:trop_lag}, we introduce the notion of tropical Lagrangian multi-sections over a complete fan $\Sigma$ on $N_{\bb{R}}\cong\bb{R}^n$. A tropical Lagrangian multi-section $\bb{L}$ over $\Sigma$ is a branched covering map $\pi:(L,\Sigma_L,\mu)\to(N_{\bb{R}},\Sigma)$ of connected cone complexes\footnote{In \cite{Suen_TP2}, we assume the domain $L$ is a topological manifold. We extend the definition here by allowing $L$ to be a cone complex, which is not necessarily a manifold.} ($\mu:\Sigma_L\to\bb{Z}_{>0}$ is the weight or multiplicity map) together with a piecewise linear function $\varphi:L\to\bb{R}$. 
    We will introduce three more concepts, namely, \emph{combinatorial union}, \emph{combinatorially indecomposability} and \emph{combinatorially equivalence}. These concepts allow us to break down a tropical Lagrangian multi-section into ``indecomposable" components. Moreover, these components enjoy some nice properties, for instance, the ramification locus of a combinatorially indecomposable tropical Lagrangian multi-section lies in the codimension 2 strata of $(L,\Sigma_L$) (Proposition \ref{prop:com_inde_separated}). Such indecomposability is also related to indecomposability of toric vector bundles as we will see in Section \ref{sec:reconsturction} (Theorem \ref{thm:L_cid_E_id})
    .

	In Section \ref{sec:bundle_lag}, we follow \cite{branched_cover_fan} to associate a tropical Lagrangian multi-section $\bb{L}_{\cu{E}}$ to a toric vector bundle $\cu{E}$ on $X_{\Sigma}$. Section \ref{sec:reconsturction} will be devoted to the converse. Namely, given a tropical Lagrangian multi-section $\bb{L}$ over a complete fan $\Sigma$, we would like to construct a toric vector bundle on $X_{\Sigma}$. We call this the \emph{reconstruction problem}. One should not expect $\bb{L}$ can completely determine a toric vector bundle due to its discrete nature, and Payne has already proved in \cite{branched_cover_fan} that $\bb{L}_{\cu{E}}$ only determines the total equivariant Chern class of $\cu{E}$. Therefore, we need to introduce some continuous data (Definition \ref{def:compatible}), which are the linear algebra data given by Kaneyama \cite{Kaneyama_classification}. The set of all such data on $\bb{L}$ modulo gauge equivalence will be denoted by $\cu{K}(\bb{L})$.
	
	A fundamental question that this paper would like to answer is: When is $\cu{K}(\bb{L})\neq\emptyset$? In Section \ref{sec:SYZ}, we give a ``SYZ-mirror-symmetric" approach to solve this problem. First of all, SYZ mirror symmetry \cite{SYZ} suggests that if a symplectic manifold admits a Lagrangian torus fibration, its complex mirror is obtained by taking the dual torus fibration. Furthermore, the SYZ program also suggests that holomorphic vector bundles are mirror to Lagrangian multi-sections. Given a Lagrangian multi-section whose underlying covering map is unbranched, its SYZ transform was defined in \cite{CS_SYZ_imm_Lag, LYZ}. However, the covering map can be branched over the base of the SYZ fibration. The SYZ program then suggests we first construct the \emph{semi-flat bundle}, which is obtained by the usual SYZ transform with the branch locus removed. However, the semi-flat bundle would receive non-trivial monodromies around those fibers above the branch locus and thus cannot be extended to the whole mirror space. To perform extension, we need to cancel these monodromies by remembering the ramification locus. The SYZ program suggests that the ramification locus should be remembered by the holomorphic disks bounded by the multi-section and certain SYZ fibers. The exponentiation of the generating function of these holomorphic disks is the so-called \emph{wall-crossing automorphism}. A good local example was given by Fukaya \cite{Fukaya_asymptotic_analysis}, Example 4.4. Moreover, he also pointed out in \cite{Fukaya_asymptotic_analysis}, Section 6.4 that, when the rank is 2, the semi-flat bundle needs to be twisted by a non-trivial local system in order to carry out the monodromy cancellation process.
	
	Going back to our tropical world, we restrict our attention to combinatorially indecomposable tropical Lagrangian multi-sections. This assumption implies the ramification locus is contained in the codimension 2 stratum $L^{(n-2)}$ of $(L,\Sigma_L)$ (Proposition \ref{prop:com_inde_separated}). Following the idea of the SYZ program and Fukaya's proposal, the reconstruction program should consist of two steps. The first step is to equip $L\backslash L^{(n-2)}$ with a suitable $\bb{C}^{\times}$-local system $\cu{L}$. Then we construct in Section \ref{sec:sf}, the \emph{semi-flat mirror bundle} $\cu{E}^{sf}(\bb{L},\cu{L})$ of $(\bb{L},\cu{L})$, which is a rank $r$ toric vector bundle defined on the 1-skeleton
	$$X_{\Sigma}^{(1)}:=\bigcup_{\tau\in\Sigma(n-1)}X_{\tau}$$
	of $X_{\Sigma}$. In general, the semi-flat mirror bundle cannot be extended to $X_{\Sigma}$ due to the presence of monodromies of $\pi:L\to N_{\bb{R}}$ around the branch locus $S\subset N_{\bb{R}}$. In order to cancel these monodromies, we will introduce a set of local automorphisms $\Theta:=\{\Theta_{\tau}(\omega')\}_{\tau\in\Sigma(n-1),\omega'\subset S}$ in Section \ref{sec:wall} to correct the transition maps of $\cu{E}^{sf}(\bb{L},\cu{L})$ so that it can be extended to $X_{\Sigma}$. If there exists a $\bb{C}^{\times}$-local system $\cu{L}$ on $L\backslash L^{(n-2)}$ and a collection of factors $\Theta$ that satisfy the \emph{consistency condition} (Definition \ref{def:consistent}), the tropical Lagrangian multi-section is called \emph{unobstructed} (Definition \ref{def:unobs} and see Remark \ref{rmk:unobstructed} for the terminology). Being unobstructed allows us to define a 1-cocycle $\{G_{\sigma_1\sigma_2}\}_{\sigma_1,\sigma_2\in\Sigma(n)}$ and gives a toric vector bundle $\cu{E}(\bb{L},\cu{L},\Theta)$ over $X_{\Sigma}$. It turns out that all Kaneyama data arise from this construction.
    \begin{theorem}[=Theorem \ref{thm:bundle_unobstructed}]
    Suppose $\bb{L}$ is combinatorially indecomposable and admits a Kaneyama data ${\bf{g}}$. Then there exists a $\bb{C}^{\times}$-local system $\cu{L}$ on $L\backslash L^{(n-2)}$ and consistent $\Theta$ such that $\cu{E}(\bb{L},\cu{L},\Theta)=\cu{E}(\bb{L},{\bf{g}})$.
    \end{theorem}
	The factors $\{\Theta_{\tau}(\omega')\}$ should be thought of as wall-crossing automorphisms as described above, which are responsible for Maslov index 0 holomorphic disks bounded by a Lagrangian multi-section and certain fibers of the torus fibration $T^*N_{\bb{R}}/M\to N_{\bb{R}}$. Hence our reconstruction program can be regarded as a ``tropical SYZ transform".

    In the last section, Section \ref{sec:rank2}, we apply our ``SYZ construction" to study the unobstructedness of combinatorially indecomposable tropical Lagrangian multi-sections of rank $2$ over a complete fan on $N_{\bb{R}}\cong\bb{R}^2$. First of all, not all such objects are unobstructed (Example \ref{eg:separated_obstructed}). Therefore, we need extra conditions to guarantee unobstructedness. We will define a \emph{slope condition} (Definition \ref{def:slope_condition}), which is completely determined by the combinatorics of the piecewise linear function $\varphi:L\to\bb{R}$ of $\bb{L}$. It turns out this combinatorial condition completely determines the obstruction of $\bb{L}$.
    
	\begin{theorem}[=Theorem \ref{thm:rank2_unobstructed}]
    A combinatorially indecomposable rank $2$ tropical Lagrangian multi-section $\bb{L}$ over a 2-dimensional complete fan $\Sigma$ is unobstructed if and only if it satisfies the slope condition.
    \end{theorem}
    
    From the proof of Theorem \ref{thm:rank2_unobstructed}, we can deduce an interesting inequality, bounding the dimension of moduli spaces of toric vector bundles with fixed equivariant Chern classes by the number of rays in $\Sigma$.
    
    \begin{corollary}[=Corollary \ref{cor:dim}]
    If $\bb{L}$ is a combinatorially indecomposable rank $2$ tropical Lagrangian multi-section, then we have the inequality $\dim_{\bb{C}}(\cu{K}(\bb{L}))\leq\#\Sigma(1)-1$.
    \end{corollary}

	\subsection*{Acknowledgment}
	The author is grateful to 	
	Kwokwai Chan and Ziming Nikolas Ma for useful discussions. The author would also like to thank Yong-Geun Oh for his interest in this work. The work of Y.-H. Suen was supported by IBS-R003-D1.
	
	\section{Toric varieties and toric vector bundles}\label{sec:toric}
	
	We first recall some basics in toric geometry. Standard references are \cite{Cox_book, Dan_toric_geometry, Fulton_book}. Throughout the whole article, we denote by $N$ a rank $n$ lattice and $M:=Hom_{\bb{Z}}(N,\bb{Z})$ the dual lattice. We also set $N_{\bb{R}}:=N\otimes_{\bb{Z}}\bb{R}$ and $M_{\bb{R}}:=M\otimes_{\bb{Z}}\bb{R}$. A \emph{fan} $\Sigma$ in $N_{\bb{R}}$ is a collection of rational strictly convex cones in $N_{\bb{R}}$ such that
	\begin{enumerate}
	    \item if $\sigma\in\Sigma$ and $\tau\subset\sigma$ is a face, then $\tau\in\Sigma$ and
	    \item if $\sigma_1,\sigma_2\in\Sigma$, then $\sigma_1\cap\sigma_2\in\Sigma$.
	\end{enumerate}
	Denote by $\Sigma(k)$ the collection of all $k$-dimensional cones in $\Sigma$. For each cone $\sigma\in\Sigma$, one can associate the corresponding dual cone $\sigma^{\vee}$ in $M_{\bb{R}}$, which is defined by
	$$\sigma^{\vee}:=\{x\in M_{\bb{R}}:\inner{x,\xi}\geq 0,\forall\xi\in\sigma\}.$$
	It is also a strictly convex rational cone. For $\tau\subset\sigma$, we have $\sigma^{\vee}\subset\tau^{\vee}$. Define
	$$U(\sigma):=\Spec{\bb{C}[\sigma^{\vee}\cap M]}.$$
	There is a $(\bb{C}^{\times})^n$-action on $U(\sigma)$, given by
	$$\lambda\cdot z^m:=\lambda^mz^m,$$
	for $m\in\sigma^{\vee}\cap M$. For $\tau\subset\sigma$, we have an open embedding $U(\tau)\to U(\sigma)$. The toric variety $X_{\Sigma}$ associated to $\Sigma$ is defined to be the direct limit
	$$X_{\Sigma}:=\lim_{\longrightarrow}U(\sigma).$$
	The $(\bb{C}^{\times})^n$-actions on affine charts agree and so induce a $(\bb{C}^{\times})^n$-action on $X_{\Sigma}$.

	\begin{definition}
    Let $X_{\Sigma}$ be an $n$-dimensional toric variety. A vector bundle $\cu{E}$ on $X_{\Sigma}$ is called \emph{toric} if the $(\bb{C}^{\times})^n$-action on $X_{\Sigma}$ lifts to an action on $\cu{E}$ which is linear on fibers. Equivalently (see \cite{Kaneyama_classification}), for each $\lambda\in(\bb{C}^{\times})^n$, there is a vector bundle isomorphism $\lambda^*\cu{E}\cong\cu{E}$ covering the identity of $X_{\Sigma}$.
	\end{definition}

	Given a toric vector bundle $\cu{E}$ on $X_{\Sigma}$, the $(\bb{C}^{\times})^n$-action constrains the transition maps of $\cu{E}$. Let $G_{\sigma}:\cu{E}|_{U(\sigma)}\to U(\sigma)\times\bb{C}^r$ be an equivariant trivialization and
	$$G_{\sigma_1\sigma_2}:=G_{\sigma_2}\circ G_{\sigma_1}^{-1}:U(\sigma_1\cap\sigma_2)\times\bb{C}^r\to U(\sigma_1\cap\sigma_2)\times\bb{C}^r$$
	be the transition map from the affine chart $U(\sigma_1)$ to the chart $U(\sigma_2)$. We can always choose the trivialization $G_{\sigma}:\cu{E}|_{U(\sigma)}\to U(\sigma)\times\bb{C}^r$ so that $(\bb{C}^{\times})^n$ acts diagonally on fibers, that is, the action on $\bb{C}[\sigma^{\vee}\cap M]\otimes_{\bb{C}}\bb{C}[t_1,\dots,t_r]$ is of form
	$$\lambda\cdot (z^m,t_1,\dots,t_r)=(\lambda^mz^m,\lambda^{m^{(1)}(\sigma)}t_1,\dots,\lambda^{m^{(r)}(\sigma)}t_r),$$
	for some $m^{(1)}(\sigma),\dots,m^{(r)}(\sigma)\in M$. Since this action extends to $X_{\Sigma}$, we must have
	$$G_{\sigma_1\sigma_2}^{(\alpha\beta)}(z)=g_{\sigma_1\sigma_2}^{(\alpha\beta)}z^{m^{(\alpha)}(\sigma_1)-m^{(\beta)}(\sigma_2)},$$
	for some $g_{\sigma_1\sigma_2}^{(\alpha\beta)}\in\bb{C}$ so that $g_{\sigma_1\sigma_2}^{(\alpha\beta)}\neq 0$ only if $m^{(\alpha)}(\sigma_1)-m^{(\beta)}(\sigma_2)\in(\sigma_1\cap\sigma_2)^{\vee}\cap M$.

	\section{Tropical Lagrangian multi-sections}\label{sec:trop_lag}
	
	In this section, we introduce the notion of tropical Lagrangian multi-sections. We begin by reviewing some basics about cone complexes. We follow \cite{branched_cover_fan} with some small notation changes.
	
	\begin{definition}[\cite{branched_cover_fan}, Definition 2.1]
	A \emph{cone complex} consists of a topological space $X$ together with a finite collection $\Sigma$ of closed subsets of $X$ and for each $\sigma\in\Sigma$, a finitely generated subgroup $M(\sigma)$ of the group of continuous functions on $\sigma$, satisfying the following conditions:
	\begin{enumerate}
	    \item The natural map $\phi_{\sigma}:\sigma\to(M(\sigma)\otimes_{\bb{Z}}\bb{R})^{\vee}$ given by
	    $$x\mapsto\left(u\mapsto u(x)\right)$$
	    maps $\sigma$ homeomorphically onto a convex rational polyhedral cone.
	    \item The preimage of any face of $\phi_{\sigma}(\sigma)$ is an element of $\Sigma$ and $M(\tau)=\{m|_{\tau}\,|\,m\in M(\sigma)\}$.
	    \item The topological space $X$ admits the following decomposition
	    $$X=\bigsqcup_{\sigma\in\Sigma}\Int(\sigma),$$
	    where $\Int(\sigma)$ denotes the relative interior of $\sigma$.
	\end{enumerate}
	A cone complex $(X,\Sigma)$ is said to be \emph{connected} if the topological space $X$ is connected. The space of \emph{piecewise linear functions} on $(X,\Sigma)$ is defined to be
	$$PL(X,\Sigma):=\{\varphi:X\to\bb{R}\,|\,\varphi|_{\sigma}\in M(\sigma),\forall \sigma\in\Sigma\}.$$
	\end{definition}
	
	\begin{remark}
	The connected components of $X$ are parametrized by minimal cones in $\Sigma$. See \cite{branched_cover_fan}, Remark 2.6.
	\end{remark}

	\begin{definition}[\cite{branched_cover_fan}, Definition 2.9]
	A \emph{morphism} of cone complexes $f:(X',\Sigma_{X'})\to(X,\Sigma_X)$ is a continuous map $f:X'\to X$ such that for any $\sigma'\in\Sigma_{X'}$, there exists $\sigma\in\Sigma$ such that $f(\sigma')\subset\sigma$ and $f^*M(\sigma)\subset M(\sigma')$. 
	\end{definition}
	
	\begin{definition}[\cite{branched_cover_fan}, Definition 2.16]
	A \emph{weighted cone complex} consists of a cone complex $(X,\Sigma)$ together with a function $\mu:X\to\bb{Z}_{>0}$ such that for any $\sigma\in\Sigma$, $\mu|_{\Int(\sigma)}$ is constant. We simply write $\mu(\sigma)$ for $\mu|_{\Int(\sigma)}$.
	\end{definition}
	
	If $(X',\Sigma_{X'})$ is weighted by $\mu$, for a surjective morphism $f:(X',\Sigma_{X'})\to(X,\Sigma_X)$, we can define $Tr_f(\mu):X\to\bb{Z}_{>0}$ by
	$$Tr_f(\mu)(x):=\sum_{x'\in f^{-1}(x)}\mu(x'),$$
	called the \emph{trace of $\mu$ by $f$}.
	
	\begin{definition}[\cite{branched_cover_fan}, Definition 2.17]
    Let $(B,\Sigma)$ be a connected cone complex and $(L,\Sigma_L,\mu)$ be a connected weighted cone complex. A \emph{branched covering map} $\pi:(L,\Sigma_L,\mu)\to(B,\Sigma)$ is a surjective morphism of cone complexes such that
    \begin{enumerate}
        \item For each $\sigma'\in\Sigma_L$, $\pi$ maps $\sigma$ homeomorphically to $\pi(\sigma)\in\Sigma$.
        \item For any connected open set $U\subset B$ and connected $V\subset\pi^{-1}(U)$, the function $Tr_{\pi|_V}(\mu):U\to\bb{Z}_{>0}$ is constant.
    \end{enumerate} 
    The morphism $\pi:(L,\Sigma_L,\mu)\to(B,\Sigma)$ is said to be \emph{ramified along $\tau'\in\Sigma_L$} if $\mu(\tau')>1$. The number $Tr_{\pi}(\mu)$ is called the degree of $\pi:(L,\Sigma_L,\mu)\to(B,\Sigma)$. The subset
    $$S':=S'(\bb{L}):=\bigcup_{\tau'\in\Sigma_L:\mu(\tau')>1}\tau'\subset L$$
    is called the \emph{ramification locus} of $\pi$ and $S:=S(\bb{L}):=\pi(S')$ is called the \emph{branch locus} of $\pi$.
	\end{definition}
	
	\begin{definition}\label{def:po}
	Let $\pi_1:(L_1,\Sigma_{L_1},\mu_1)\to(B,\Sigma),\pi_2:(L_2,\Sigma_{L_2},\mu_2)\to(B,\Sigma)$ be branched covering maps of the same degree. We write $\pi_1\leq\pi_2$ if there exists a surjective morphism of cone complexes $f:(L_2,\Sigma_{L_2})\to(L_1,\Sigma_{L_1})$ such that $\pi_1\circ f=\pi_2$ and $Tr_f(\mu_2)=\mu_1$.
	\end{definition}

	\begin{definition}[\cite{branched_cover_fan}, Definition 2.26]
	A branched covering map $\pi:(L,\Sigma_L,\mu)\to(B,\Sigma)$ is called \emph{maximal} if it is maximal with respective to the partial ordering given in Definition \ref{def:po}.
	\end{definition}

	Given a cone complex $(L,\Sigma_L)$, we define
	$$L^{(n-k)}:=\bigcup_{\tau'\in\Sigma_L:\codim(\tau')=k}\tau'\subset L,$$
	the codimension $k$ stratum of $(L,\Sigma_L)$. Payne showed in \cite{branched_cover_fan}, Proposition 2.30 that if $\Sigma$ is a complete fan in $N_{\bb{R}}$, the ramification locus of any maximal branched covering map $\pi:(L,\Sigma_L,\mu)\to(N_{\bb{R}},\Sigma)$ lies in the codimension 2 stratum $L^{(n-2)}$ of $(L,\Sigma_L)$.	Now we focus on $B=N_{\bb{R}}\cong\bb{R}^n$ and $\Sigma$ is a complete fan on $N_{\bb{R}}$. In this case, $B$ carries a natural affine structure and $\Sigma$ turns $(B,\Sigma)$ into a cone complex. If $\pi:(L,\Sigma_L,\mu)\to(B,\Sigma)$ is a branched covering map, then for any $\sigma'\in\Sigma_L(n)$, we have $\pi^*M=\pi^*M(\sigma)=M(\sigma')$ as $\pi|_{\sigma'}:\sigma'\to\sigma$ is an isomorphism. Hence we can identity $M(\sigma')$ with $M$ via $\pi^*$ naturally. We can then define
	$$Lin(L):=\{f\in C^0(L,\bb{R}):\exists m\in M\text{ such that }f|_{\sigma'}=m,\forall \sigma'\in\Sigma_L\},$$
	to be the space of linear function on $L$. It is clear that $Lin(L)\subset PL(L,\Sigma_L)$. Moreover, as $L$ is assumed to be connected, it is clear that $Lin(L)=Lin(B)=M$.
	
	\begin{definition}\label{def:TLMS}
	 Let $\Sigma$ be a complete fan on $N_{\bb{R}}$. A \emph{tropical Lagrangian multi-section of rank $r$ over $\Sigma$} is a quintuple $\bb{L}:=(L,\Sigma_L,\mu,\pi,\varphi)$, where
	 \begin{enumerate}
	     \item $(L,\Sigma_L)$ is a connected cone complex weighted by $\mu$.
	     \item $\pi:(L,\Sigma_L,\mu)\to(N_{\bb{R}},\Sigma)$ is a branched covering map such that $Tr_{\pi}(\mu)=r$.
	     \item $\varphi$ is a piecewise linear function on $(L,\Sigma_L)$.
	 \end{enumerate}
	 The number $r$ is called the \emph{rank} of $\bb{L}$ and is denoted by $rk(\bb{L})$. The underlying branched covering map of $\bb{L}$ is denoted by $\ul{\bb{L}}$. A tropical Lagrangian multi-section $\bb{L}$ is said to be \emph{maximal} if $\ul{\bb{L}}$ is maximal.
	\end{definition}
	
	\begin{remark}
	In \cite{Suen_TP2}, the author provided a definition of tropical Lagrangian multi-sections over integral affine manifolds with singularities whose domain of the branched covering map is a topological manifold. While in \cite{CMS_k3bundle}, the authors gave a definition of tropical Lagrangian multi-sections over 2-dimensional integral affine manifolds with singularities \emph{equipped with polyhedral decomposition}, where they also assumed the domain is also a topological manifold equipped with a polyhedral decomposition that is compatible with the covering map. Of course, if we restrict our attention to the case where the affine manifold is $\bb{R}^2$ with polyhedral decomposition being a fan $\Sigma$, Definition \ref{def:TLMS} extends Definition 3.6 in \cite{CMS_k3bundle} because we don't assume $L$ is a topological manifold here.
	\end{remark}
	
	\begin{remark}
	In \cite{Abouzaid09}, Abouzaid used the terminology ``tropical Lagrangian section" to stand for an honest Lagrangian section of the torus fibration ${\rm{Log}}:(\bb{C}^{\times})^n\to\bb{R}^n$. The term ``tropical" in this paper stands for a combinatorial/discrete replacement for Lagrangian multi-sections, which are supposed to be mirror to vector bundles on $X_{\Sigma}$. However, it is not hard to show that a tropical Lagrangian section ($r=1$) in our combinatorial sense always produces a tropical Lagrangian section in the sense of Abouzaid by smoothing the piecewise linear function $\varphi:|\Sigma|\to\bb{R}$ suitably. Thus our definition is somehow a generalization of Abouzaid's one. Nevertheless, we apologize for any possible confusion with the use of the terminology here.
	\end{remark}
	
	\begin{definition}\label{def:order}
	Let $\bb{L}_1,\bb{L}_2$ be tropical Lagrangian multi-sections of the same rank. We write $\bb{L}_1\leq\bb{L}_2$ if $\ul{\bb{L}}_1\leq\ul{\bb{L}}_2$ via some $f$ such that $f^*\varphi_1=\varphi_2$.
	\end{definition}
	

	\begin{definition}
	Let $\bb{L}_1,\bb{L}_2$ be tropical Lagrangian multi-sections over a fan $\Sigma$. We write $\bb{L}_2\sim_c\bb{L}_1$ if $rk(\bb{L}_1)=rk(\bb{L}_2)$ and there exists a tropical Lagrangian multi-section $\bb{L}$ over $\Sigma$ such that $\bb{L}\leq\bb{L}_i$, for all $i=1,2$. We say $\bb{L}_1$ is \emph{combinatorially equivalent} to $\bb{L}_2$ if there exists a sequence of tropical Lagrangian multi-sections $\bb{L}_1',\bb{L}_2',\dots,\bb{L}_k'$ such that $\bb{L}_1'=\bb{L}_1,\bb{L}_k'=\bb{L}_2$ and $\bb{L}_{i+1}'\sim_c\bb{L}_i'$, for all $i=1,\dots,k-1$.
	\end{definition}
	
	\begin{remark}
	The relation $\sim_c$ is only reflexive and symmetric. The notion of combinatorially equivalence is the transitive closure of $\sim_c$ and hence, an equivalence relation.
	\end{remark}

	Now we define an important class of tropical Lagrangian multi-sections.

	\begin{definition}\label{def:separated}
	A tropical Lagrangian multi-section $\bb{L}=(L,\Sigma_L,\mu,\pi,\varphi)$ is said to be \emph{$k$-separated} if it satisfies the following condition: For any $\tau\in\Sigma(k)$ and distinct lifts $\tau^{(\alpha)},\tau^{(\beta)}\in\Sigma_L(k)$ of $\tau$, we have $\varphi|_{\tau^{(\alpha)}}\neq\varphi|_{\tau^{(\beta)}}$. Note that $k$-separability implies $K$-separability for all $K\geq k$. A tropical Lagrangian multi-section is said to be \emph{separated} if it is $1$-separated.
	\end{definition}
	
	\begin{remark}
	Definition \ref{def:separated} holds vacuously for all rank 1 tropical Lagrangian multi-sections.
	\end{remark}
	
	We can always ``separate" a tropical Lagrangian multi-section in the following sense.
	
	\begin{proposition}\label{prop:seperated}
	For any tropical Lagrangian multi-section $\bb{L}$ over $\Sigma$, there exists a separated tropical Lagrangian multi-section $\bb{L}_{sep}$ over $\Sigma$ such that $\bb{L}_{sep}\leq\bb{L}$. In particular, every tropical Lagrangian multi-section is combinatorially equivalent to a separated one.
	\end{proposition}
	\begin{proof}
	    We define a cone complex $(L_{sep},\Sigma_{sep}')$ as follows. Let $\sigma\in\Sigma$. Two lifts $\sigma^{(\alpha)},\sigma^{(\beta)}\in\Sigma$ of $\sigma$ are identified if and only if $\varphi|_{\sigma^{(\alpha)}}=\varphi|_{\sigma^{(\beta)}}$. We denote the quotient map $L\to L_{sep}$ by $q$. The set of cones is given by
	    $$\Sigma_{sep}':=\{q(\sigma')\,|\,\sigma'\in\Sigma_L\}.$$
	    The projection map $\pi:L\to N_{\bb{R}}$ factors through $q$ and hence descends to a projection $\pi_{sep}:L_{sep}\to N_{\bb{R}}$. Define $\mu_{sep}:=Tr_q(\mu)$. It is clear that $\pi_{sep}:(L_{sep},\Sigma_{sep}',\mu_{sep})\to(N_{\bb{R}},\Sigma)$ is a branched covering map. We define $\varphi_{sep}:L_{sep}\to\bb{R}$ by
	    $$\varphi_{sep}|_{q(\sigma')}=\varphi|_{\sigma'}.$$
	    It is clear that $\varphi_{sep}|_{q(\sigma')}$ is independent of the choice of $\sigma'\in\Sigma_L$ and $\varphi_{sep}$ is continuous. It also follows from construction that $q^*\varphi_{sep}=\varphi$. Hence $\bb{L}_{sep}\leq\bb{L}$.
	\end{proof}
	
	\begin{example}
	Given a tropical Lagrangian multi-section $\bb{L}$ as shown in Figure \ref{fig:separation}, its canonical separation $\bb{L}_{sep}$ is given by gluing $\sigma_0^{(1)},\sigma_0^{(2)}$ over $\sigma_0$.
	\begin{figure}[H]
		\centering
		\includegraphics[width=140mm]{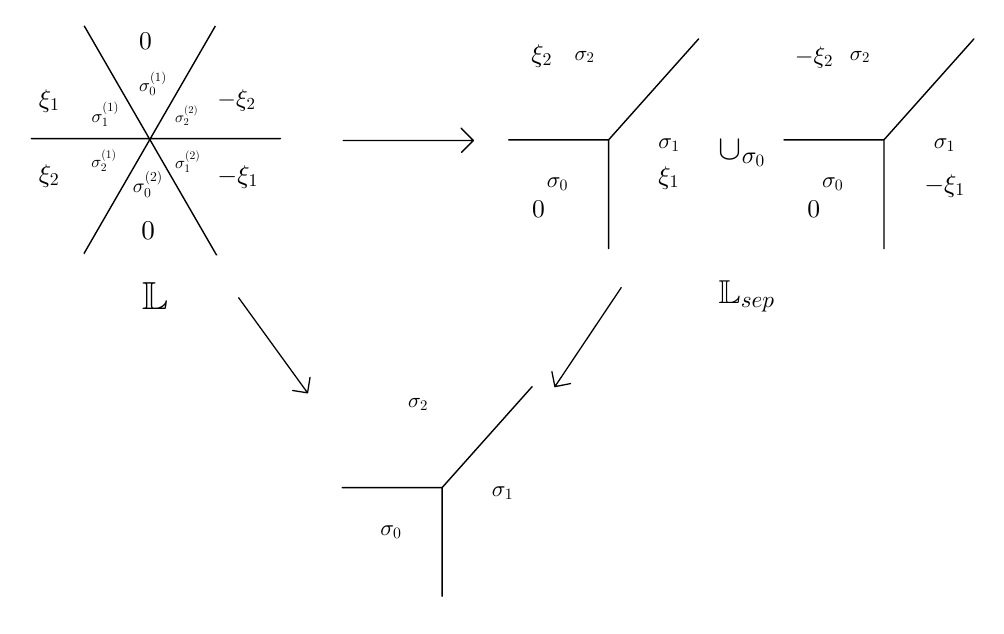}
	    \caption{The canonical separation $\bb{L}_{sep}$ of $\bb{L}$.}
		\label{fig:separation}
    \end{figure}
	\end{example}
	
	\begin{definition}
	The tropical Lagrangian multi-section $\bb{L}_{sep}$ constructed in the proof of Proposition \ref{prop:seperated} is called the \emph{canonical separation of $\bb{L}$}.
	\end{definition}
	
	\begin{construction}
    There are three natural operations on tropical Lagrangian multi-sections. As we will see in Proposition \ref{prop:operations}, they correspond to algebraic operations of toric vector bundles.
    \begin{enumerate}
        \item Given $\bb{L}=(L,\Sigma_L,\mu,\pi,\varphi)$, we put $-\bb{L}:=(L,\Sigma_L,\mu,\pi,-\varphi)$, called the \emph{dual of $\bb{L}$}.
        \item Given two tropical Lagrangian multi-sections $\bb{L}_1,\bb{L}_2$ with rank $r_1,r_2$, respectively, we can construct another tropical Lagrangian multi-section $\bb{L}_1\cup_c\bb{L}_2$ by gluing $\tau'\in\Sigma_1',\tau''\in\Sigma_2'$ whenever $\pi_1(\tau')=\pi(\tau'')=\tau$ and $\varphi_1|_{\tau'}=\varphi_1|_{\tau''}$. The domain of $\bb{L}_1\cup_c\bb{L}_2$ is denoted by $L_1\cup_cL_2$ and the quotient map $L_1\sqcup L_2\to L_1\cup_cL_2$ is denoted by $q$. The set of cones is given by
    	$$\Sigma_1'\cup_c\Sigma_2':=\{q(\sigma')\,|\,\sigma'\in\Sigma_1'\sqcup\Sigma_2'\}$$
	    and the multiplicity map is given by
    	$$(\mu_1\cup_c\mu_2)(\ol{\sigma}'):=\sum_{\sigma_1'\in\Sigma_1':q(\sigma_1')=\ol{\sigma}'}\mu_1(\sigma_1')+\sum_{\sigma_2'\in\Sigma_2':q(\sigma_2')=\ol{\sigma}'}\mu_2(\sigma_2')$$
    	In particular, the rank of $\bb{L}_1\cup_c\bb{L}_2$ is $r_1+r_2$. Finally, the piecewise linear function is given by
    	$$(\varphi_1\cup_c\varphi_2)|_{\ol{\sigma}'}=\begin{cases}
    	\varphi_1|_{\sigma_1'} & \text{ if }q(\sigma_1')=\ol{\sigma}'\in q(\Sigma_1'),\\
    	\varphi_2|_{\sigma'} & \text{ if }q(\sigma_2')=\ol{\sigma}'\in q(\Sigma_2').
    	\end{cases}$$
    	It follows from the definition of $q$ that $\varphi_1\cup_c\varphi_2$ is well-defined and continuous. We call the tropical Lagrangian multi-section $\bb{L}_1\cup_c\bb{L}_2$ the \emph{combinatorial union of $\bb{L}_1,\bb{L}_2$}.
    	\item We define the tropical Lagrangian multi-section $\bb{L}_1\times_{\Sigma}\bb{L}_2$ of rank $r_1r_2$ with domain $L_1\times_{|\Sigma|}L_2$, the set of cones $\Sigma_1'\times_{\Sigma}\Sigma_2'$, the multiplicity map
    	$$\sigma_1'\times\sigma_2'\mapsto\mu_1(\sigma_1)\mu_2(\sigma_2)$$
    	and the projection $\sigma_1\times_{\sigma}\sigma_2\mapsto\sigma$. The piecewise linear function is given by
    	$$(x_1,x_2)\mapsto\varphi_1(x_1)+\varphi_2(x_2).$$
    	Finally, denote the canonical separation of $\bb{L}_1\times_{|\Sigma|}\bb{L}_2$ by $\bb{L}_1\times_c\bb{L}_2$, called the \emph{combinatorial fiber product of $\bb{L}_1,\bb{L}_2$}.
    \end{enumerate}
    Note that $\bb{L}_1\cup_c\bb{L}_2,\bb{L}_1\times_c\bb{L}_2$ are always separated by construction.
	\end{construction}
	
	\begin{definition}
	Let $\bb{L},\bb{L}_1,\bb{L}_2$ be tropical Lagrangian multi-sections over $\Sigma$. We say $\bb{L}$ is \emph{combinatorially decomposable} by $\bb{L}_1,\bb{L}_2$ if $\bb{L}$ is combinatorially equivalent to $\bb{L}_1\cup_c\bb{L}_2$. A tropical Lagrangian multi-section is said to be \emph{combinatorially indecomposable} if it is not combinatorially decomposable for all pairs of $\bb{L}_1,\bb{L}_2$.
	\end{definition}
	
	Every tropical Lagrangian multi-section can be combinatorially decomposed into a union of indecomposable ones. However, such decomposition is not unique most of the time.
	
	\begin{example}\label{ex:P2}
	Figure \ref{fig:P2} shows a combinatorial indecomposable tropical multi-section over the fan of $\bb{P}^2$. It is also separated as the piecewise linear function has different slopes along distinct lifts of every ray. This tropical Lagrangian multi-section is in fact the associated branched covering map of cone complexes of $T_{\bb{P}^2}$. See \cite{branched_cover_fan}.
	\begin{figure}[H]
		\centering
		\includegraphics[width=120mm]{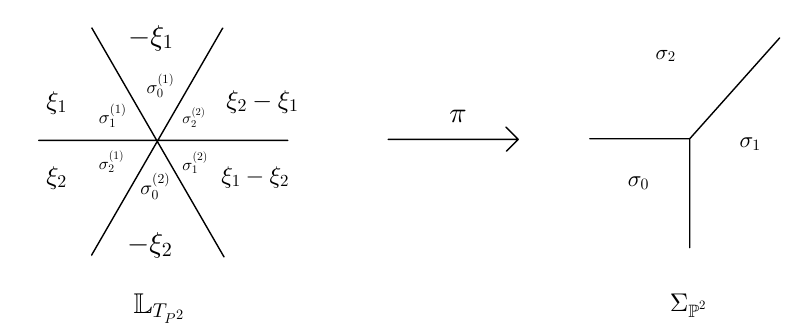}
	    \caption{A combinatorially indecomposable tropical Lagrangian multi-section over the fan of $\bb{P}^2$}
		\label{fig:P2}
    \end{figure}
	\end{example}
	\begin{example}\label{ex:F1}	
    Figure \ref{fig:P2} shows a combinatorial indecomposable tropical Lagrangian multi-section over the fan $\Sigma_{\bb{F}_1}$ of the Hirzebruch surface $\bb{F}_1$. The notation $\cup_0$ stands for gluing the two cone complexes (both are $(\bb{R}^2,\Sigma_{\bb{F}_1})$, but decorated by two different piecewise linear functions) on the left at the origin $0\in N_{\bb{R}}$. Again, it is easy to see that this tropical Lagrangian multi-section is also separated.
    \begin{figure}[H]
		\centering
		\includegraphics[width=160mm]{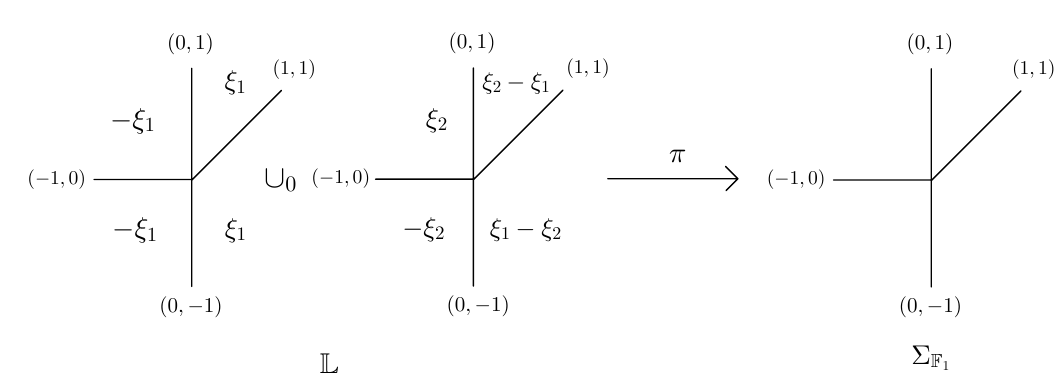}
	    \caption{A combinatorially decomposable tropical Lagrangian multi-section over the fan of $\bb{F}_1$}
		\label{fig:F_1}
    \end{figure}
	\end{example}
	
	As Example \ref{ex:P2} suggests, there is a relation between combinatorially indecomposablity and separability.
	
	\begin{proposition}\label{prop:com_inde_separated}
    	Suppose $\bb{L}$ is combinatorially indecomposable. Then $\bb{L}$ is $(n-1)$-separated and the ramification locus of $\pi:L\to N_{\bb{R}}$ lies in the codimension 2 strata $L^{(n-2)}$ of $(L,\Sigma_L)$. When $\dim(N_{\bb{R}})=2$, the converse is true with the stronger assumption that $\bb{L}$ is maximal. 
	\end{proposition}
	\begin{proof}
	   We first prove that combinatorially indecomposability implies $(n-1)$-separability under the assumption $S'(\bb{L})\subset L^{(n-2)}$. Suppose $\bb{L}$ is not $(n-1)$-separated, that is, there exists $\tau\in\Sigma(n-1)$ and distinct lifts $\tau^{(\alpha)},\tau^{(\beta)}\in\Sigma_L$ such that $\varphi|_{\tau^{(\alpha)}}=\varphi|_{\tau^{(\beta)}}$. Choose a loop $\gamma:[0,1]\to N_{\bb{R}}\backslash S(\bb{L})$ so that $\gamma(0)=\gamma(1)\in\Int(\tau)$ and it goes into the interior of each maximal cone once and transverse to the codimension 1 strata. By concatenating $\gamma$ with itself and using the path lifting lemma, we obtain a lift $\gamma':[0,1]\to L\backslash S'(\bb{L})$ of $\gamma$ so that $\gamma'(0)\in\Int(\tau^{(\alpha)})$ and $\gamma'(1)\in\Int(\tau^{(\beta)})$. Let
	   \begin{align*}
	       \Sigma_{\gamma'}^{(1)}:=&\{\sigma'\in\Sigma_L:\Int(\sigma')\cap\gamma'\neq\emptyset\},\\
	       \til{L}_{\gamma'}^{(1)}:=&\bigcup_{\sigma'\in\Sigma_{\gamma'}(n)}\sigma'\subset L.
	   \end{align*}
	   Then there is a cone complex $(L_{\gamma'}^{(1)},\Sigma_{L_{\gamma'}^{(1)}})$ obtained by gluing $\tau^{(\alpha)},\tau^{(\beta)}$. Denote the quotient map by $q_1:\til{L}_{\gamma'}^{(1)}\to L_{\gamma'}^{(1)}$. By considering $\Sigma_{\gamma'}^{(2)}=(\Sigma_L\backslash\Sigma_{\gamma'}^{(1)})\cup\{\tau^{(\alpha)},\tau^{(\beta)}\}$ and gluing $\tau^{(\alpha)},\tau^{(\beta)}$, we obtain another cone complex $(L_{\gamma'}^{(2)},\Sigma_{L_{\gamma'}^{(2)}})$ and a quotient map $q_2$. There are two obvious projections $\pi_{\gamma'}^{(i)}:L_{\gamma'}^{(i)}\to N_{\bb{R}}$. We take $\mu_{L_{\gamma'}^{(i)}}:=Tr_{q_i}(\mu|_{\til{L}_{\gamma'}^{(i)}})$ to make $\pi_{\gamma'}^{(i)}$'s into branch covering maps. The function $\varphi|_{\til{L}_{\gamma'}^{(i)}}$ descends to $L_{\gamma'}^{(i)}$ and turn them into two tropical Lagrangian multi-sections $\bb{L}_{\gamma'}^{(1)},\bb{L}_{\gamma'}^{(2)}$. It is then clear that $\bb{L}=\bb{L}_{\gamma'}^{(1)}\cup_c\bb{L}_{\gamma'}^{(2)}$.
	   
	   Now we handle the general case. Suppose $S'(\bb{L})\not\subset L^{(n-2)}$. Then there is a codimension 1 cone $\tau\in\Sigma(n-1)$ such that $\tau\subset S(\bb{L})$. Pass to a cover $f:\bb{L}'\to\bb{L}$ such that $S'(\bb{L}')$ lies in the codimension 2 strata of $(L',\Sigma_{L'})$. Then $\tau$ has two distinct lifts $\tau^{(\alpha)},\tau^{(\beta)}\in\Sigma_{L'}(n-1)$ such that $f^*\varphi|_{\tau^{(\alpha)}}=f^*\varphi|_{\tau^{(\beta)}}$. Hence $\bb{L}'$ is not $(n-1)$-separated and hence combinatorially decomposable. But $\bb{L}'$ is combinatorially equivalent to $\bb{L}$ and so $\bb{L}$ is also combinatorially decomposable.
	   
	   For the converse, note that $1$-separability of $\bb{L}$ implies any covering morphism of the form $\bb{L}\to\bb{L}'$ is an isomorphism. Indeed, if $f:\bb{L}\to\bb{L}'$ is not injective, there exists distinct $\tau^{(\alpha)},\tau^{(\beta)}\in\Sigma_L(1)$ so that $f(\sigma^{(\alpha)})=f(\sigma^{(\beta)})$. This implies $\varphi|_{\tau^{(\alpha)}}=\varphi|_{\tau^{(\beta)}}$. As $\tau^{(\alpha)}\neq\tau^{(\beta)}$, this contradicts separability. However, maximality of $\bb{L}$ also implies all covering morphism of the form $\bb{L}'\to\bb{L}$ is an isomorphism. Therefore, if $\bb{L}$ is combinatorially decomposable, say by $\bb{L}_1,\bb{L}_2$, then $\bb{L}\cong\bb{L}_1\cup_c\bb{L}_2$, which violate maximality.
	\end{proof}
	
	\begin{remark}
	The converse of Proposition \ref{prop:com_inde_separated} is not true without the maximality assumption. For example, let $\Sigma$ be the fan of $\bb{P}^2$ and $\varphi_0,\varphi_1$ be the piecewise linear functions correspond to $\cu{O}_{\bb{P}^2},\cu{O}_{\bb{P}^2}(D_1+D_2-2D_0)$, where $D_0,D_1,D_2$ are invariant divisors. Then $\bb{L}_i:=(N_{\bb{R}},\Sigma,1,id_{N_{\bb{R}}},\varphi_i)$, $i=0,1$ are tropical Lagrangian multi-sections. Then it is easy to see that $\bb{L}_0\cup_c\bb{L}_1$ is separated with the zero cone being the only ramification point. It is obvious that $\bb{L}_0\cup_c\bb{L}_1$ is combinatorially decomposable by $\bb{L}_0,\bb{L}_1$.
	\end{remark}

	\subsection{From  toric  vector  bundles  to  tropical  Lagrangian  multi-sections}\label{sec:bundle_lag}
	
	Let $X_{\Sigma}$ be the associated toric variety of $\Sigma$. Given a rank $r$ toric vector bundle $\cu{E}$ on $X_{\Sigma}$, we can associate a rank $r$ tropical Lagrangian multi-section $\bb{L}_{\cu{E}}$ over $\Sigma$ by following the construction in \cite{branched_cover_fan}.
	
	Let $\sigma\in\Sigma$ and $U(\sigma)$ be the affine toric variety corresponding to $\sigma$. The toric vector bundle splits equivariantly on $U(\sigma)$ as
	$$\cu{E}|_{U(\sigma)}\cong\bigoplus_{m(\sigma)\in\textbf{m}(\sigma)}\cu{L}_{m(\sigma)},$$
	where $\textbf{m}(\sigma)\subset M(\sigma):=M/(\sigma^{\perp}\cap M)$ is a multi-set and $\cu{L}_{m(\sigma)}$ is the line bundle corresponds to the linear function $m(\sigma)\in M(\sigma)$. We define $\bb{L}_{\cu{E}}$ as follows. Let
	Let $|\Sigma|\to\Sigma$ be the map given by mapping $x\in|\Sigma|$ to the unique cone $\sigma\in\Sigma$ such that $x\in \Int(\sigma)$. Equip $\Sigma$ with the quotient topology. Define
	$$\Sigma_{\cu{E}}:=\{(\sigma,m(\sigma))\,|\,\sigma\in\Sigma, m(\sigma)\in\textbf{m}(\sigma)\}$$
	and $\Sigma_{\cu{E}}\to\Sigma$ be the projection 
	$$(\sigma,m(\sigma))\mapsto\sigma.$$
	We emphasize that although $\textbf{m}(\sigma)$ is a multi-set, $\Sigma_{\cu{E}}$ is not. Equip $\Sigma_{\cu{E}}$ a poset structure
	$$(\sigma_1,m(\sigma_1))\leq (\sigma_2,m(\sigma_2))\Longleftrightarrow\sigma_1\subset\sigma_2\text{ and }m(\sigma_2)|_{\sigma_1}=m(\sigma_1)$$
	and equip it with the poset topology, namely, a subset $K\subset\Sigma_L$ is closed if and only if $$\{(\sigma_1,m(\sigma_1))\,|\,(\sigma_1,m(\sigma_1))\leq(\sigma_2,m(\sigma_2))\}\subset K$$
	for all $(\sigma_2,m(\sigma_2))\in K$. Define
	$$L_{\cu{E}}:=|\Sigma|\times_{\Sigma}\Sigma_{\cu{E}}.$$
	Let the set of cones on $L_{\cu{E}}$ be $\Sigma\times_{\Sigma}\Sigma_{\cu{E}}\cong\Sigma_{\cu{E}}$. The multiplicity $\mu_{\cu{E}}:L_{\cu{E}}\to\bb{Z}_{>0}$ is defined by
	$$\mu_{\cu{E}}(\sigma,m(\sigma)):=\text{number of times that }m(\sigma)\text{ appears in }{\bf{m}}(\sigma).$$
	The projection map $\pi_{\cu{E}}:L_{\cu{E}}\to|\Sigma|$ then induces a rank $r$ branched covering map of cone complexes $\pi_{\cu{E}}:(L_{\cu{E}},\Sigma_{\cu{E}},\mu_{\cu{E}})\to(N_{\bb{R}},\Sigma)$. The piecewise linear function $\varphi_{\cu{E}}:L_{\cu{E}}\to\bb{R}$ is tautologically given by
	$$\varphi_{\cu{E}}|_{(\sigma,m(\sigma))}:=\pi_{\cu{E}}^*m(\sigma).$$
	This gives a tropical Lagrangian multi-section $\bb{L}_{\cu{E}}:=(L_{\cu{E}},\Sigma_{\cu{E}},\mu_{\cu{E}},\pi_{\cu{E}},\varphi_{\cu{E}})$.
 
    \begin{proposition}\label{prop:L_E_separated}
    The tropical Lagrangian multi-section $\bb{L}_{\cu{E}}$ is separated.
    \end{proposition}
    \begin{proof}
        By construction, if $\omega^{(\alpha)},\omega^{(\beta)}\in\Sigma_{\cu{E}}$ are distinct lifts of some $\omega\in\Sigma$, then $\varphi_{\cu{E}}|_{\omega^{(\alpha)}}\neq\varphi_{\cu{E}}|_{\omega^{(\alpha)}}$. In particular, slopes on different codimension 1 cones are different.
    \end{proof}
    
    \begin{proposition}\label{prop:operations}
    Let $\cu{E},\cu{E}_1,\cu{E}_2$ be toric vector bundles on $X_{\Sigma}$. Then we have
    \begin{enumerate}
        \item $\bb{L}_{\cu{E}^*}=-\bb{L}_{\cu{E}}$.
        \item $\bb{L}_{\cu{E}_1\oplus\cu{E}_2}=\bb{L}_{\cu{E}_1}\cup_c\bb{L}_{\cu{E}_2}$.
        \item $\bb{L}_{\cu{E}_1\otimes\cu{E}_2}=\bb{L}_{\cu{E}_1}\times_c\bb{L}_{\cu{E}_2}$.
    \end{enumerate}
    \end{proposition}
    \begin{proof}
        They follow from the induced equivariant structure:
        \begin{align*}
            \lambda\cdot f:=&\,f(\lambda^{-1}\cdot v),\\
            \lambda\cdot(v_1\oplus v_2):=&\,(\lambda\cdot v_1)\oplus(\lambda\cdot v_2),\\
            \lambda\cdot(v_1\otimes v_2):=&\,(\lambda\cdot v_1)\otimes(\lambda\cdot v_2),
        \end{align*}
        where $f\in\cu{E}^*,v\in\cu{E},v_1\in\cu{E}_1,v_2\in\cu{E}_2$.
    \end{proof}
	
	The assignment $\cu{E}\mapsto\bb{L}_{\cu{E}}$ is not injective as the following example shows.
	
	\begin{example}\label{eg:TP2}
	Consider the toric vector bundles $\cu{E}_1:=\bigoplus\limits_{i=1}^2\cu{O}_{\bb{P}^2}(D_i)$ and $\cu{E}_2:=T_{\bb{P}^2}\oplus\cu{O}_{\bb{P}^2}$. Via the Euler sequence
	$$0\to\cu{O}_{\bb{P}^2}\to\bigoplus_{i=1}^2\cu{O}_{\bb{P}^2}(D_i)\to T_{\bb{P}^2}\to 0,$$
	$\cu{E}_1,\cu{E}_2$ share the same equivariant Chern class and hence $\bb{L}_{\cu{E}_1}=\bb{L}_{\cu{E}_2}$ by Proposition 3.4 of \cite{branched_cover_fan}. This example also shows that combinatorially indecomposable components are not unique. Indeed, we have $\bb{L}_{\cu{E}_1}=\bb{L}_{\cu{O}_{\bb{P}^2}(D_0)}\cup_c\bb{L}_{\cu{O}_{\bb{P}^2}(D_1)}\cup_c\bb{L}_{\cu{O}_{\bb{P}^2}(D_2)}$, $\bb{L}_{\cu{E}_2}=\bb{L}_{\cu{O}_{\bb{P}^2}}\cup_c\bb{L}_{T_{\bb{P}^2}}$ and it is easy to see that $\bb{L}_{T_{\bb{P}^2}}$ is maximal and separated, hence combinatorially indecomposable.
	\end{example}

	\section{Kaneyama's classification via SYZ-type construction}\label{sec:reconsturction}
	
	\subsection{Kaneyama's classification}
	
	We first rewrite Kaneyama's classification result in terms of the language of tropical Lagrangian multi-sections. By doing so, some properties of toric vector bundles can be read off from the tropical Lagrangian multi-sections.
	
	In \cite{Kaneyama_classification}, Kaneyama classified toric vector bundles by both combinatorial and linear algebra data. We can rewrite and refine these data in terms of the language of tropical Lagrangian multi-sections. Let $\bb{L}=(L,\Sigma_L,\mu,\pi,\varphi)$ be a tropical Lagrangian multi-section over $\Sigma$. For a maximal cone $\sigma'\in\Sigma_L$, we use the notation $m(\sigma')$ to denote the slope of $\varphi$ on $\sigma'$, which is an element in $M$. We also count lifts of a maximal cone with multiplicities (recall that each cone $\sigma'\in\Sigma_L$ has a multiplicity $\mu(\sigma')$).
	
    \begin{definition}\label{def:compatible}
    Let $\bb{L}$ be a tropical Lagrangian multi-section of rank $r$ over $\Sigma$. A \emph{Kaneyama data of $\bb{L}$} is a collection ${\bf{g}}:=\{g_{\sigma_1\sigma_2}\}_{\sigma_1,\sigma_2\in\Sigma_L(n)}\subset GL(r,\bb{C})$ such that
    \begin{enumerate}
        \item [(G1)] For any $\sigma\in\Sigma(n)$, we have $g_{\sigma\sigma}=Id$
        \item [(G2)] For any $\sigma_1,\sigma_2\in\Sigma(n)$, the $(\alpha,\beta)$-entry $g_{\sigma_1^{(\alpha)}\sigma_2^{(\beta)}}$ of $g_{\sigma_1\sigma_2}$ is non-zero only if $\sigma_1^{(\alpha)}\cap\sigma_2^{(\beta)}\neq\emptyset$ and  $$m(\sigma_1^{(\alpha)})-m(\sigma_2^{(\beta)})\in (\sigma_1\cap\sigma_2)^{\vee}\cap M.$$
	    \item [(G3)] For any $\sigma_1,\sigma_2,\sigma_3\in\Sigma(n)$, we have
	    $$g_{\sigma_1\sigma_2}g_{\sigma_2\sigma_3}=g_{\sigma_1\sigma_3}.$$
    \end{enumerate}
    We denote by $\til{\cu{K}}(\bb{L})$ the set of Kaneyama data on $\bb{L}$. Two Kaneyama data ${\bf{g}},{\bf{g}}'\in\til{\cu{K}}(\bb{L})$ is said to be \emph{equivalent} if for any $\sigma\in\Sigma(n)$, there exists $h_{\sigma}:=(h_{\sigma^{(\alpha)}\sigma^{(\beta)}})\in GL(r,\bb{C})$ such that
    \begin{enumerate}
        \item [(H1)] $h_{\sigma^{(\alpha)}\sigma^{(\beta)}}\neq 0$ only if
        $$m(\sigma^{(\alpha)})-m(\sigma^{(\beta)})\in\sigma^{\vee}\cap M$$
        \item [(H2)] For any $\sigma_1,\sigma_2\in\Sigma(n)$,
        $$h_{\sigma_2}g_{\sigma_1\sigma_2}=g_{\sigma_1\sigma_2}'h_{\sigma_1}.$$
    \end{enumerate}
    We denote by $\cu{K}(\bb{L})$ the set of equivalence classes of Kaneyama data on $\bb{L}$.
    \end{definition}
    
    \begin{remark}
    In Kaneyama's work \cite{Kaneyama_classification}, Pages 74-75, conditions (i) and (i') there are equivalent to continuity of $\varphi$, conditions (ii) is equivalent to (G1),(G2),(G3) and Condition (iii) is equivalent to (H1),(H2).
	\end{remark}
	
	\begin{theorem}[A Reformulation of \cite{Kaneyama_classification}, Theorem 4.2]\label{thm:Lag_bundle}
    Let $\bb{L}$ be a tropical Lagrangian multi-section over $\Sigma$. If $\bb{L}$ admits a Kaneyama data ${\bf{g}}$, then there is a toric vector bundle $\cu{E}(\bb{L},{\bf{g}})$ over $X_{\Sigma}$ such that $\bb{L}_{\cu{E}(\bb{L},{\bf{g}})}\leq\bb{L}$. Two Kaneyama data ${\bf{g}},{\bf{g}}'\in\til{\cu{K}}(\bb{L})$ are equivalent if and only if $\cu{E}(\bb{L},{\bf{g}})\cong\cu{E}(\bb{L},{\bf{g}}')$ as toric vector bundles.
    \end{theorem}
    \begin{proof}
        The $(\bb{C}^{\times})^n$-action on the toric vector bundle $\cu{E}_{\sigma}=\bigoplus\limits_{\alpha=1}^r\cu{L}_{m(\sigma^{(\alpha)})}$ on $U(\sigma)$ is given by
        \begin{equation}\label{eqn:equ_str}
            \lambda\cdot (p,1(\sigma^{(\alpha)})):=(\lambda\cdot p,\lambda^{m(\sigma^{(\alpha)})}1(\sigma^{(\alpha)})),
        \end{equation}
        where $p\in U(\sigma)$ and $1(\sigma^{(\alpha)})$ is an equivariant holomorphic frame of $\cu{L}_{m(\sigma^{(\alpha)})}$. It is a straightforward checking that this action is compatible with the transition maps $$G_{\sigma_1\sigma_2}:1(\sigma_1^{(\alpha)})\mapsto\sum_{\beta=1}^rg_{\sigma_1^{(\alpha)}\sigma_2^{(\beta)}}z^{m(\sigma_1^{(\alpha)})-m(\sigma_2^{(\beta)})}1(\sigma_2^{(\beta)}).$$
        To prove that $\bb{L}_{\cu{E}(\bb{L},{\bf{g}})}\leq\bb{L}$, we define
        $$f_{\sigma'}:\sigma'\to\pi(\sigma')\times\{m(\sigma')\}.$$
        By continuity of $\varphi$, $\{f_{\sigma'}\}_{\sigma'\in\Sigma_L}$ can be glued to a continuous map $f:L\to L_{\cu{E}(\bb{L},{\bf{g}})}$ which maps cones in $\Sigma_L$ to cones in $\Sigma_{\cu{E}(\bb{L},{\bf{g}})}$ homeomorphically. By definition, $f^*\varphi_{\cu{E}(\bb{L},{\bf{g}})}=\varphi$ and for any $\sigma\times\{m(\sigma)\}$, we have
        $$Tr_f(\mu)(\sigma\times\{m(\sigma)\})=\sum_{\sigma':\varphi|_{\sigma'}=m(\sigma)}\mu(\sigma')=\#\{m\in{\bf{m}}(\sigma):m=m(\sigma)\}=\mu_{\cu{E}(\bb{L},{\bf{g}})}(\sigma\times\{m(\sigma)\}).$$
        Hence $\bb{L}_{\cu{E}(\bb{L},{\bf{g}})}\leq\bb{L}$ via $f$. The last assertion follows from Condition (iii) in \cite{Kaneyama_classification}.
    \end{proof}

	Suppose $\bb{L}$ admits a Kaneyama data ${\bf{g}}$. The composition
	$$\bb{L}\mapsto\cu{E}(\bb{L},{\bf{g}})\mapsto\bb{L}_{\cu{E}(\bb{L},{\bf{g}})}$$
	may not be the identity map. For instance, suppose $\pi:L\to N_{\bb{R}}$ is a 2-fold cover conjugate to the square map $z\mapsto z^2$ on $\bb{C}$. Let $\Sigma$ be the fan of $\bb{P}^2$. Then there is a natural collection of cones $\Sigma'$ on $L$. Equip $L$ with the 0 function. Then, the Kaneyama data ${\bf{g}}$ there gives a rank $2$ toric vector bundle, which is just $\cu{O}_{\bb{P}^2}^{\oplus 2}$ with the trivial equivariant structure. But it is clear that the associated tropical Lagrangian multi-section of $\cu{O}_{\bb{P}^2}^{\oplus 2}$ is given by $(N_{\bb{R}},\Sigma,\mu,id_{N_{\bb{R}}},0)$, with $\mu(\sigma)=2$. Nevertheless, the map $\pi:L\to N_{\bb{R}}$ gives a branched covering of cone complexes that preserve the function. More generally, we have the following
	
	\begin{theorem}\label{thm:combin_equiv}
	Let $\bb{L}_1,\bb{L}_2$ be tropical Lagrangian multi-sections of the same rank $r$. If $\bb{L}_1,\bb{L}_2$ are combinatorially equivalent, then there exists a bijection $f_*:\cu{K}(\bb{L}_1)\to\cu{K}(\bb{L}_2)$ such that $\cu{E}(\bb{L}_1,{\bf{g}}_1)\cong\cu{E}(\bb{L}_2,f_*({\bf{g}}_1))$ as toric vector bundles. Conversely, if $\cu{E}(\bb{L}_1,{\bf{g}}_1)\cong\cu{E}(\bb{L}_2,{\bf{g}}_2)$ for some Kaneyama data, then $\bb{L}_1$ is combinatorially equivalent to $\bb{L}_2$.
	\end{theorem}
	\begin{proof}
	   It suffices to prove that if $\bb{L}_2\leq\bb{L}_1$ via some $f$, then any Kaneyama data of $\bb{L}_1$ gives a Kaneyama data of $\bb{L}_2$ such that their associated toric vector bundles are the same and vice versa. Let $\sigma_1',\sigma_2'\in\Sigma_1'(n)$ be maximal cones. By the assumption $f^*\varphi_2=\varphi_1$, we have
	   $$m(f(\sigma_1'))-m(f(\sigma_2'))=m(\sigma_1')-m(\sigma_2').$$
	   Moreover, counting with multiplicity, $f$ induces a permutation of the index set $\{1,\dots,r\}$, which parametrizes lifts of a maximal cell. Thus if ${\bf{g}}$ is a Kaneyama data of $\bb{L}_1$, then we can simply define
	   $$(f_*g)_{f(\sigma_1^{(\alpha)})f(\sigma_2^{(\beta)})}:=g_{\sigma_1^{(\alpha)}\sigma_2^{(\beta)}},$$
	   where $\sigma_1^{(\alpha)},\sigma_2^{(\beta)}$ are preimage of $f(\sigma_1^{(\alpha)}),f(\sigma_2^{(\beta)})$ such that $m(f(\sigma_1^{(\alpha)}))=m(\sigma_1^{(\alpha)}),m(f(\sigma_2^{(\beta)}))=m(\sigma_2^{(\beta)})$. Although the lifts $\sigma_1^{(\alpha)},\sigma_2^{(\beta)}$ are not unique, the slopes are and hence $f_*g$ is well-defined. It is straightforward to check that $f_*({\bf{g}}):=\{(f_*g)_{\sigma_1\sigma_2}\}$ is a Kaneyama data for $\bb{L}_2$ and any two choices of preimages above differ by a permutation of the equivariant frame $\{1(\sigma^{(\alpha)})\}_{\alpha=1}^r$ and the torus action is preserved. It is then easy to see that there is an isomorphism $\cu{E}(\bb{L}_1,{\bf{g}})\cong\cu{E}(\bb{L}_2,f_*({\bf{g}}))$ of toric vector bundles. By pulling back, Kaneyama data on $\bb{L}_2$ induces a Kaneyama data on $\bb{L}_1$. Modulo equivalence, we obtain the desired bijection. The converse follows from Theorem \ref{thm:Lag_bundle}.
	\end{proof}
	
	\begin{remark}
    Theorem \ref{thm:combin_equiv} has the following analog in mirror symmetry. Non-Hamiltonian equivalent Lagrangian branes in a symplectic manifold may give rise to the same mirror object as they can still be isomorphic in the derived Fukaya category. For example, in \cite{CS_SYZ_imm_Lag}, Example 5.5 gives a Lagrangian immersion and a Lagrangian embedding in a symplectic 2-torus that shares the same mirror sheaf.
	\end{remark}

	\begin{proposition}\label{prop:embedding}
	Suppose $\bb{L}=\bb{L}_1\cup_c\bb{L}_2$. Then there exists an embedding $\cu{K}(\bb{L}_1)\times\cu{K}(\bb{L}_2)\to\cu{K}(\bb{L})$.
	\end{proposition}
	\begin{proof}
	    The embedding is given by taking the direct sum of matrices.
	\end{proof}
	
	Every tropical Lagrangian multi-section can be combinatorially decomposed into combinatorially indecomposable ones. By Proposition \ref{prop:embedding}, to obtain Kaneyama data on a general tropical Lagrangian multi-section, it suffices to consider its combinatorially indecomposable components.
	
	\begin{theorem}\label{thm:L_cid_E_id}
	If $\bb{L}$ is combinatorially indecomposable, then $\cu{E}(\bb{L},{\bf{g}})$ is indecomposable for any Kaneyama data ${\bf{g}}$ of $\bb{L}$. The converse is also true if $\bb{L}$ can be decomposed into a combinatorial union of two tropical Lagrangian multi-sections $\bb{L}_1,\bb{L}_2$ that admits Kaneyama data.
	\end{theorem}
	\begin{proof}
	   If $\cu{E}(\bb{L},{\bf{g}})$ is decomposable for some ${\bf{g}}$, say by $\cu{E}_1,\cu{E}_2$, then $\bb{L}_{\cu{E}(\bb{L},{\bf{g}})}=\bb{L}_{\cu{E}_1}\cup_c\bb{L}_{\cu{E}_2}$. Since $\bb{L}_{\cu{E}(\bb{L},{\bf{g}})}\leq\bb{L}$ by Theorem \ref{thm:Lag_bundle}, $\bb{L}$ is also combinatorially decomposable. Conversely, suppose $\bb{L}=\bb{L}_1\cup_c\bb{L}_2$ for some unobstructed $\bb{L}_1,\bb{L}_2$. Let ${\bf{g}}_1,{\bf{g}}_2$ be some Kaneyama data of $\bb{L}_1,\bb{L}_2$, respectively. Denote the image of $({\bf{g}}_1,{\bf{g}}_2)$ under the embedding $\cu{K}(\bb{L}_1)\times\cu{K}(\bb{L}_2)\to\cu{K}(\bb{L})$ by ${\bf{g}}$. Then we have $\cu{E}(\bb{L},{\bf{g}})=\cu{E}(\bb{L}_1,{\bf{g}}_1)\oplus\cu{E}(\bb{L}_2,{\bf{g}}_2)$.
	\end{proof}
	
	Since sections ($r=1$) always admit Kaneyama data, we have the following
	
	\begin{corollary}\label{cor:indecomposable_rk2}
	A rank 2 tropical Lagrangian multi-section $\bb{L}$ is combinatorially indecomposable if and only if $\cu{E}(\bb{L},{\bf{g}})$ is indecomposable for any Kaneyama data ${\bf{g}}$ of $\bb{L}$.
	\end{corollary}
	
	\begin{remark}\label{rmk:counter_example}
	The converse of Theorem \ref{thm:L_cid_E_id} or Corollary \ref{cor:indecomposable_rk2} is not true if we just ask for $\cu{E}(\bb{L},{\bf{g}})$ to be indecomposable \emph{for some ${\bf{g}}$}. For instance, take any indecomposable toric vector bundle $\cu{E}$ that contains a toric subbundle. Then $\bb{L}_{\cu{E}}$ is combinatorially decomposable since $\cu{E}$ fits into an exact sequence of toric vector bundles. A concrete example is given by the tangent bundle of the Hirzebruch surface $\bb{F}_1$, which is indecomposable. But it contains a line bundle as a toric subbundle. See Corollary 4.1.2 of \cite{toric_bundle_Bott_tower}.
	\end{remark}
	
	\subsection{A mirror symmetric approach}\label{sec:SYZ}
	
	Now we go into one of the main themes of this paper. We would like to interpret Kaneyama's result in terms of mirror symmetry. We assume from now on all tropical Lagrangian multi-sections are combinatorially indecomposable and hence by Proposition \ref{prop:com_inde_separated}, they are separated and the ramification locus $S'$ always lies in the codimension 2 strata of $(L,\Sigma_L)$.

    \subsubsection{The semi-flat bundle}\label{sec:sf}
    
    For a tropical multi-section $\bb{L}=(L,\Sigma_L,\mu,\pi,\varphi)$, we have denoted the ramification locus by $S'$ and the branch locus by $S$. Both of them are assumed to be contained in the codimension 2 strata. We define the 1-skeleton of $X_{\Sigma}$:
    $$X_{\Sigma}^{(1)}:=\bigcup_{\tau\in\Sigma(n-1)}X_{\tau}=X_{\Sigma}\big\backslash\bigcup_{\dim(\omega)\leq n-1}U(\omega).$$
    The semi-flat bundle is a locally free sheaf on $X_{\Sigma}^{(1)}$. To construct it, we first provide a good open cover for $L\backslash L^{(n-2)}$. For each $\sigma'\in\Sigma_L(n)$, choose a small neighborhood $V_{\sigma'}\subset L\backslash L^{(n-2)}$ contains $\sigma'\backslash L^{(n-2)}$ such that $V_{\sigma_1'}\cap V_{\sigma_2'}\neq\emptyset$ if and only if $\sigma_1'\cap\sigma_2'\in\Sigma_L(n-1)$. See Figure \ref{fig:nbhd_V}.
    \begin{figure}[H]
		\centering
		\includegraphics[width=45mm]{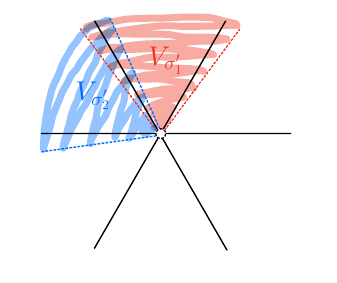}
	    \caption{The space $L\backslash L^{(n-2)}$ and the neighborhoods $V_{\sigma_1'},V_{\sigma_2'}$.}
		\label{fig:nbhd_V}
    \end{figure}
    Choose any $\bb{C}^{\times}$-local system $\cu{L}$ on $L\backslash L^{(n-2)}$. Denote the transition map on $V_{\sigma_1'}\cap V_{\sigma_2'}$ by
    $$1_{\sigma_1'}\mapsto g_{\sigma_1'\sigma_2'}^{sf}1_{\sigma_2'},$$
    where $\sigma_2'\in\Sigma_L(n)$ is the unique lift of $\sigma_2$ such that $\sigma_1'\cap\sigma_2'\in\Sigma_L(n-1)$. For a cone $\sigma\in\Sigma$, let $V(\sigma):=U(\sigma)\cap X_{\Sigma}^{(1)}$. If $\omega\subset S$, then $V(\omega)=\emptyset$. Thus, $\{V(\sigma)\}_{\sigma\in\Sigma(n)}$ forms an open cover of $X_{\Sigma}^{(1)}$ such that if $\sigma_1\cap\sigma_2\in\Sigma(n-1)$, we have $\emptyset\neq V(\sigma_1\cap\sigma_2)\subset X_{\Sigma}^{(1)}$. For a maximal cone $\sigma\in\Sigma(n)$, we put
    $$\cu{E}_{\sigma}:=\bigoplus_{\alpha=1}^r\cu{L}_{m(\sigma^{(\alpha)})},$$
    which is a toric vector bundle defined on $U(\sigma)$. For $\sigma_1,\sigma_2\in\Sigma(n)$ so that $\sigma_1\cap\sigma_2\in\Sigma(n-1)$, we define $G_{\sigma_1\sigma_2}^{sf}:\cu{E}_{\sigma_1}|_{V(\sigma_1\cap\sigma_2)}\to\cu{E}_{\sigma_2}|_{V(\sigma_1\cap\sigma_2)}$ by
    $$G_{\sigma_1\sigma_2}^{sf}:1(\sigma_1^{(\alpha)})\mapsto g_{\sigma_1^{(\alpha)}\sigma_2^{(\beta)}}^{sf}z^{m(\sigma_1^{(\alpha)})-m(\sigma_2^{(\beta)})}1(\sigma_2^{(\beta)}),$$
    where $\sigma_2^{(\beta)}$ is uniquely determined by the conditions $\emptyset\neq\sigma_1^{(\alpha)}\cap\sigma_2^{(\beta)}\in\Sigma_L(n-1)$ and $\pi(\sigma_2^{(\beta)})=\sigma_2$. Since we have no triple intersections, $\{g_{\sigma_1'\sigma_2'}^{sf}\}$ immediately satisfies the cocycle condition.
    
    \begin{definition}
    Let $\bb{L}=(L,\Sigma',\mu,\pi,\varphi)$ be a tropical Lagrangian multi-section over $\Sigma$. Equip $L\backslash L^{(n-2)}$ with a $\bb{C}^{\times}$-local system $\cu{L}$. The vector bundle $\cu{E}^{sf}(\bb{L},\cu{L})$ is called the \emph{semi-flat bundle of $(\bb{L},\cu{L})$}.
    \end{definition}
    
    \subsubsection{Wall-crossing factors}\label{sec:wall}
    
    After constructing the semi-flat bundle $\cu{E}^{sf}(\bb{L},\cu{L})$ of $(\bb{L},\cu{L})$, we would like to extend $\cu{E}^{sf}(\bb{L},\cu{L})$ to the whole space $X_{\Sigma}$. To do this, we may need to correct $G_{\sigma_1\sigma_2}^{sf}$ by certain factors. Let $\tau\in\Sigma(n-1)$ and $\sigma_1,\sigma_2\in\Sigma(n)$ be the unique maximal cones so that $\sigma_1\cap\sigma_2=\tau$. For each $\omega'\in\Sigma_L$, we define a bundle map $N_{\tau}(\omega'):\cu{E}_{\sigma_1}|_{U(\tau)}\to\cu{E}_{\sigma_1}|_{U(\tau)}$ so that with respective to the frame $\{1(\sigma_1^{(\alpha)})\}$, the $(\alpha,\beta)$-entry is given by
	$$N_{\tau,\sigma_1}^{(\alpha\beta)}(\omega'):=\begin{cases}
	n_{\tau,\sigma_1}^{(\alpha\beta)}(\omega')z^{m(\sigma_1^{(\alpha)})-m(\sigma_1^{(\beta)})} & \text{ if }\omega'\subset\sigma_1^{(\alpha)}\cap\sigma_1^{(\beta)},\alpha\neq\beta\text{ and }m(\sigma_1^{(\alpha)})-m(\sigma_1^{(\beta)})\in\tau^{\vee}\cap M\\
	0 & \text{ otherwise},
	\end{cases}$$
	for some $n_{\tau,\sigma_1}^{(\alpha\beta)}(\omega')\in\bb{C}$. Note that $N_{\tau,\sigma_1}(\omega')=0$ if $\omega'\not\subset S'$. Put
    $$S_{\tau}'(\sigma_1):=\{\sigma_1^{(\alpha)}\cap\sigma_1^{(\beta)}\,|\,m(\sigma_1^{(\alpha)})-m(\sigma_1^{(\beta)})\in(\sigma_1\cap\sigma_2)^{\vee}\cap M\}.$$
    By assumption, cones in $S_{\tau}'$ are of codimension $\geq 2$. Moreover, there is a natural bijection $S_{\tau}'(\sigma_1)\cong S_{\tau}'(\sigma_2)$. Indeed, for $\sigma_1^{(\alpha)}\cap\sigma_1^{(\beta)}\in S_{\tau}'$, there exists unique $\sigma_2^{(\alpha')},\sigma_2^{(\beta')}\in\Sigma_L(n)$ so that $\sigma_1^{(\alpha)}\cap\sigma_2^{(\alpha')},\sigma_1^{(\beta)}\cap\sigma_2^{(\beta')}\in\Sigma_L(n-1)$. Then
    $$m(\sigma_2^{(\alpha')})-m(\sigma_2^{(\beta')})=(m(\sigma_2^{(\alpha')})-m(\sigma_1^{(\alpha)}))-(m(\sigma_2^{(\beta')})-m(\sigma_1^{(\beta)}))+(m(\sigma_1^{(\alpha)})-m(\sigma_1^{(\beta)})).$$
    The first two terms of the right-hand side are in $\tau^{\perp}\cap M$ by continuity and the last term is in $\tau^{\vee}\cap M$ by definition. Hence $\sigma_2^{(\alpha')}\cap\sigma_2^{(\beta')}\in S_{\tau}'$. As $\sigma_2^{(\alpha')},\sigma_2^{(\beta')}$ are uniquely determined by $\sigma_1^{(\alpha)},\sigma_1^{(\beta)}$ and vice versa, the assignment $\sigma_1^{(\alpha)}\cap\sigma_1^{(\beta)}\mapsto\sigma_2^{(\alpha')}\cap\sigma_2^{(\beta')}$ gives the desired bijection. Now we define
	$$N_{\tau,\sigma_1}:=\sum_{\omega'\in S_{\tau}'(\sigma_1)}N_{\tau,\sigma_1}(\omega').$$
	
	\begin{remark}
	Note that the change of frame from $\{1(\sigma_1^{(\alpha)})\}$ to $\{1(\sigma_2^{(\alpha)})\}$ is compatible with the bijection $S_{\tau}'(\sigma_1)\cong S_{\tau}'(\sigma_2)$, namely, by choosing suitable $n_{\tau,\sigma_2}^{(\alpha\beta)}(\omega')$, we have $G_{\sigma_1\sigma_2}^{sf}\circ N_{\tau,\sigma_1}(\omega')\circ G_{\sigma_2\sigma_1}^{sf}=N_{\tau,\sigma_2}$. As $N_{\tau,\sigma_1}, N_{\tau,\sigma_2}$ are related by a change of frame, to simplify our notation, we simply right $N_{\tau}(\omega')$ for $N_{\tau,\sigma_1}(\omega')$ and $N_{\tau}$ for $N_{\tau,\sigma_1}$. We will also write $N_{\sigma_1\sigma_2}$ for $N_{\tau}$ when we want to emphasize the maximal cones $\sigma_1,\sigma_2$ so that $\sigma_1\cap\sigma_2=\tau$.
	\end{remark}
	
	\begin{lemma}\label{lem:N_commute}
	    For any $\tau\in\Sigma(n-1)$ and $\omega'\in\Sigma_L$, $N_{\tau}(\omega')$ is nilpotent and for any distinct lifts $\omega',\omega''\in\Sigma_L$ of $\omega\in\Sigma$, $N_{\tau}(\omega')N_{\tau}(\omega'')=0$. In particular, $[N_{\tau}(\omega'),N_{\tau}(\omega'')]=0$, for any lifts $\omega',\omega''\in\Sigma_L$ of $\omega$.
	\end{lemma}
	\begin{proof}
	    We show that $N_{\tau}(\omega')^k=0$ has zero diagonal entries, for all $k\geq 1$. The case $k=1$ is by definition. Assume $N_{\tau}(\omega')^k$ has zero diagonal entries for some $k\geq 1$. If there exists $\alpha$ such that
	    $$\sum_{\beta=1}^r(N_{\tau}(\omega')^k)^{(\alpha\beta)}N_{\tau}^{(\beta\alpha)}(\omega')\neq 0,$$
	    there must exist $\beta\neq\alpha$ such that 
	    $$(N_{\tau}(\omega')^k)^{(\alpha\beta)}N_{\tau}^{(\beta\alpha)}(\omega')\neq 0,$$
	    as both $N_{\tau}(\omega')^k,N_{\tau}(\omega')$ have zero diagonal entries. This implies both $z^{m(\sigma_1^{(\alpha)})-m(\sigma_1^{(\beta)})}$ and $z^{m(\sigma_1^{(\beta)})-m(\sigma_1^{(\alpha)})}$ are regular functions on the affine chart $U(\tau)$. Hence $z^{m(\sigma_1^{(\alpha)})-m(\sigma_1^{(\beta)})}$ must be invertible on $U(\tau)$ and so $m(\sigma_1^{(\alpha)})-m(\sigma_1^{(\beta)})\in\tau^{\perp}$, which means $m(\sigma_1^{(\alpha)})|_{\tau}=m(\sigma_1^{(\beta)})|_{\tau}$. This violates $(n-1)$-separability. Hence $N_{\tau}(\omega')^{k+1}$ has zero diagonal entries too. By induction, we are done. 
	    
	    For the last part, we have
	    $$\sum_{\beta=1}^rn_{\tau}^{(\alpha\beta)}(\omega')n_{\tau}^{(\beta\gamma)}(\omega'')z^{m(\sigma_1^{(\alpha)})-m(\sigma_1^{(\gamma)})}.$$
	    This sum is non-zero only if $\omega'\subset\sigma_1^{(\alpha)}\cap\sigma_1^{(\beta)}$ and $\omega''\subset\sigma_1^{(\beta)}\cap\sigma_1^{(\gamma)}$, for some $\beta$. Then we must have $\omega''=\omega'$ as $\sigma_1^{(\beta)}$ can only contain one lift of $\omega$.
	\end{proof}
	
	Once a choice of $\{n_{\tau,\sigma}^{(\alpha\beta)}(\omega')\}$ is fixed, Lemma \ref{lem:N_commute} allows us to define the product of matrices
	\begin{equation}\label{eqn:theta}
	    \Theta_{\tau}:=\prod_{\omega'\in S_{\tau}'}\Theta_{\tau}(\omega'):=\prod_{\omega'\in S_{\tau}'}\exp(N_{\tau}(\omega'))
	\end{equation}
	unambiguously. Moreover, we have $\det(\Theta_{\tau})=1$ and so $\Theta_{\tau}$ is invertible over $\bb{C}[\tau^{\vee}\cap M]$.
	
	\begin{definition}
	For $\omega'\in\Sigma_L$, the factors $\{\Theta_{\tau}(\omega')\}_{\tau\in\Sigma(n-1)}$ are called \emph{wall-crossing automorphisms associated to $\omega'$}.
	\end{definition}
	
	\begin{remark}
	Similar to the notation $N_{\sigma_1\sigma_2}$, we write $\Theta_{\sigma_1\sigma_2}$ for $\Theta_{\tau}$ when we want to emphasize the unique maximal cones $\sigma_1,\sigma_2$ so that $\sigma_1\cap\sigma_2=\tau$.
	\end{remark}
	
	Now for $\tau\in\Sigma(n-1)$, put
	$$G_{\sigma_1\sigma_2}:=G_{\sigma_1\sigma_2}^{sf}\circ\Theta_{\sigma_1\sigma_2},$$
	where $\sigma_1,\sigma_2\in\Sigma(n)$ are uniquely determined by $\tau=\sigma_1\cap\sigma_2$. If we express $G_{\sigma_1\sigma_2}$ in terms of the frames $\{1(\sigma_1^{(\alpha)})\},\{1(\sigma_2^{(\gamma)})\}$, we have
	$$G_{\sigma_1\sigma_2}:1(\sigma_1^{(\alpha)})\mapsto\sum_{\beta=1}^r\theta_{\tau}^{(\alpha\beta)}g_{\sigma_1^{(\beta)}\sigma_2^{(\gamma)}}^{sf}z^{m(\sigma_1^{(\alpha)})-m(\sigma_2^{(\gamma)})}1(\sigma_2^{(\gamma)}).$$
	In particular, it is easy to choose $n_{\sigma_1\sigma_2}^{(\alpha\beta)}$'s such that
	$$G_{\sigma_2\sigma_1}=G_{\sigma_1\sigma_2}^{-1}.$$
    We haven't defined $G_{\sigma_1\sigma_2}$ for general $\sigma_1,\sigma_2\in\Sigma(n)$. To do this, given any $\sigma_1,\sigma_2\in\Sigma(n)$ such that $\tau:=\sigma_1\cap\sigma_2$, we consider a sequence of maximal cones $\sigma_1=\sigma_{1'},\sigma_{2'},\dots,\sigma_{l'}=\sigma_2\in\Sigma(n)$ such that $\tau\subset\sigma_{i'}$ and $\sigma_{i'}\cap\sigma_{(i+1)'}\in\Sigma(n-1)$, for all $i$. Such a sequence always exists since the branch locus $S$ is of codimension at least $2$. Then we put
    $$G_{\sigma_1\sigma_2}:=G_{\sigma_{(l-1)'}\sigma_{l'}}|_{U(\tau)}\circ\cdots\circ G_{\sigma_{1'}\sigma_{2'}}|_{U(\tau)},$$
    which is defined on $U(\tau)$. We need to ensure $G_{\sigma_1\sigma_2}$ is independent of the choice of such a sequence of maximal cones.
    
    \begin{definition}\label{def:consistent}
    Given a combinatorially indecomposable tropical Lagrangian multi-section $\bb{L}$ and a $\bb{C}^{\times}$-local system $\cu{L}$ on $L\backslash L^{(n-2)}$, a collection of wall-crossing automorphisms $\Theta:=\{\Theta_{\tau}(\omega')\}_{\tau\in\Sigma(n-1),\omega'\subset S'}$ defined by (\ref{eqn:theta}) is said to be \emph{$\omega$-consistent} if for any cycle of maximal cones
    $$\sigma_1,\sigma_2,\dots,\sigma_l,\sigma_{l+1}=\sigma_1$$
    such that $\omega\subset\sigma_i$ and $\sigma_i\cap\sigma_{i+1}\in\Sigma(n-1)$, for all $i$, the composition
    \begin{equation}\label{eqn:composition}
        G_{\sigma_l\sigma_{l+1}}|_{U(\omega)}\circ\cdots\circ G_{\sigma_1\sigma_2}|_{U(\omega)}:\cu{E}_{\sigma_1}|_{U(\omega)}\to\cu{E}_{\sigma_1}|_{U(\omega)}
    \end{equation}
    equals to the identity map on $\cu{E}_{\sigma_1}|_{U(\omega)}$. A collection of automorphisms $\Theta$ is said to be \emph{consistent} if it is $\omega$-consistent for all $\omega\in\Sigma$.
    \end{definition}
    
    \begin{proposition}
    A collection of wall-crossing automorphisms $\Theta$ is consistent if and only if it is $\omega$-consistent for all $\omega\in\Sigma(n-2)$.
    \end{proposition}
    \begin{proof}
    Fix $\omega\in\Sigma$. For each cycle of maximal cones
    $$\sigma_1,\sigma_2,\dots,\sigma_l,\sigma_{l+1}=\sigma_1$$
    that satisfy the condition in Definition \ref{def:consistent}, there is a loop $\gamma:[0,1]\to N_{\bb{R}}\backslash|\Sigma(n-2)|$ such that $\gamma(0)=\gamma(1)\in\Int(\sigma_1)$ and intersecting the codimension 1 cones $\Int(\sigma_i\cap\sigma_{i+1})$ transversely for all $i$. Note that the corresponding composition defined by (\ref{eqn:composition}) only depends on the homotopy class of $\gamma$. As $\pi_1(N_{\bb{R}}\backslash|\Sigma(n-2)|)$ is generated by loops around codimension 2 strata of $(N_{\bb{R}},\Sigma)$, we may write $\gamma$ in terms of these generators
    $$\gamma=\gamma_1*\cdots*\gamma_k.$$
    By choosing sufficiently generic $\gamma_i$'s, each of them determines a cycle of maximal cones that satisfies the condition stated in Definition \ref{def:consistent}. As the compositions correspond to $\gamma_i$'s equal to the identity, the composition corresponds to $\gamma$ also equal to the identity. Hence codimension 2 consistency implies consistency. The converse is trivial.
    \end{proof}
    
    It is clear that if $\Theta$ is consistent, then $G_{\sigma_1\sigma_2}$ is well-defined for all $\sigma_1,\sigma_2\in\Sigma(n)$ and the cocycle condition holds on arbitrary triple intersections. Let's make the following definition.
    
    \begin{definition}\label{def:unobs}
    A combinatorially indecomposable tropical Lagrangian multi-section $\bb{L}$ is called \emph{unobstructed} if there exists a $\bb{C}^{\times}$-local system $\cu{L}$ on $L\backslash L^{(n-2)}$ and a collection of consistent wall-crossing automorphisms $\Theta$. If $\bb{L}$ is unobstructed, we denote by $\cu{E}(\bb{L},\cu{L},\Theta)$ the vector bundle associated to the data $(\bb{L},\cu{L},\Theta)$.
    \end{definition}
    
    \begin{remark}\label{rmk:unobstructed}
    The notion of (weakly) unobstructed Lagrangian submanifolds was introduced in \cite{FOOO1} and \cite{AJ} for the immersed case. The main feature of an unobstructed Lagrangian submanifolds is that its Floer cohomology is well-defined and hence defines an object in the Fukaya category. In particular, unobstructed Lagrangian submanifolds should have the corresponding mirror objects. As the existence of Kaneyama's data or the data $(\cu{L},\Theta)$ are equivalent to the existence of toric vector bundles, we should think of the tropical Lagrangian multi-section can be ``realized" by an unobstructed Lagrangian. Thus, we borrow the terminology here.
    \end{remark}
    
    In defining $G_{\sigma_1\sigma_2}^{sf}$, we have chosen a 1-cocycle to represent the local system $\cu{L}$. When $\bb{L}$ is unobstructed, $\cu{E}(\bb{L},\cu{L},\Theta)$ is independent of such choice as the following proposition shows.
    
    \begin{proposition}\label{prop:independent}
    For any isomorphism $\cu{L}'\cong\cu{L}$ of local system on $L\backslash L^{(n-2)}$, there is an isomorphism $\cu{E}(\bb{L},\cu{L},\Theta)\cong\cu{E}(\bb{L},\cu{L}',\Theta')$ of toric vector bundles, for some consistent $\Theta'$.
    \end{proposition}
    \begin{proof}
        Let $f:\cu{L}\to\cu{L}'$ be an isomorphism of local systems. It induces an isomorphism $F:\pi_*\cu{L}\to\pi_*\cu{L}'$ of rank $r$ local systems. Locally, $F$ is given by a constant matrix and thus can be regarded as a toric automorphism on a chart $U(\sigma)\subset X_{\Sigma}$. We also have
        $$F_{\sigma_2}\circ G_{\sigma_1\sigma_2}^{sf}=G_{\sigma_1\sigma_2}'^{sf}\circ F_{\sigma_1}.$$
        If $\Theta$ is a consistent data, we simply define $\Theta'$ by conjugation by $F$, that is,
        $$\Theta_{\sigma_1\sigma_2}':=F_{\sigma_1}\circ\Theta_{\sigma_1\sigma_2}\circ F_{\sigma_1}^{-1}.$$
        Then it is by definition that
        $$F_{\sigma_2}\circ G_{\sigma_1\sigma_2}=G_{\sigma_1\sigma_2}'\circ F_{\sigma_1},$$
        which means $\cu{E}(\bb{L},\cu{L},\Theta)\cong\cu{E}(\bb{L},\cu{L}',\Theta')$ as toric vector bundles.
    \end{proof}
    
    combinatorially indecomposability implies the following relation between $\cu{E}^{sf}(\bb{L},\cu{L})$ and $\cu{E}(\bb{L},\cu{L},\Theta)$. 
    
    \begin{theorem}\label{thm:restrict_to_sf}
    If $\bb{L}$ is combinatorially indecomposable, then $\Theta_{\sigma_1\sigma_2}|_{X_{\Sigma}^{(1)}}=Id$, for any $\sigma_1,\sigma_2\in\Sigma(n)$ so that $\sigma_1\cap\sigma_2\not\subset S$. In particular, if $\bb{L}$ is unobstructed, then $\cu{E}(\bb{L},\cu{L},\Theta)|_{X_{\Sigma}^{(1)}}=\cu{E}^{sf}(\bb{L},\cu{L})$.
    \end{theorem}
    \begin{proof}
        Let $\tau:=\sigma_1\cap\sigma_2\in\Sigma(n-1)$. For $m(\sigma_1^{(\alpha)})-m(\sigma_1^{(\beta)})\in\tau^{\vee}\cap M$, $(n-1)$-separability implies that there exists a ray $\rho\subset\tau$ so that
        $$(m(\sigma_1^{(\alpha)})-m(\sigma_1^{(\beta)}))(v_{\rho})>0,$$
        where $v_{\rho}$ is a generator of $\rho$. Hence $z^{m(\sigma_1^{(\alpha)})-m(\sigma_1^{(\beta)})}$ vanishes along the divisor $U(\tau)\cap X_{\rho}$ and in particular, vanishes on $U(\tau)\cap X_{\tau}$. Hence $\Theta_{\tau}|_{U(\tau)\cap X_{\tau}}=Id$ and this proves $\cu{E}(\bb{L},\cu{L},\Theta)|_{X_{\Sigma}^{(1)}}=\cu{E}^{sf}(\bb{L},\cu{L})$.
    \end{proof}

	By definition, unobstructedness implies the existence of Kaneyama data. It turns out all Kaneyama data arise from our construction.
    
    \begin{theorem}\label{thm:bundle_unobstructed}
    Suppose $\bb{L}$ is combinatorially indecomposable and admits a Kaneyama data ${\bf{g}}$. Then there exists a $\bb{C}^{\times}$-local system $\cu{L}$ on $L\backslash L^{(n-2)}$ and consistent $\Theta$ such that $\cu{E}(\bb{L},\cu{L},\Theta)=\cu{E}(\bb{L},{\bf{g}})$.
    \end{theorem}
    \begin{proof}
        The transition maps of $\cu{E}(\bb{L},{\bf{g}})$ are of form
        $$1(\sigma_1^{(\alpha)})\mapsto\sum_{\beta=1}^rg_{\sigma_1^{(\alpha)}\sigma_2^{(\beta)}}z^{m(\sigma_1^{(\alpha)})-m(\sigma_2^{(\beta)})}1(\sigma_2^{(\beta)}).$$
        Consider two distinct maximal cones $\sigma_1,\sigma_2\in\Sigma(n)$ such that $\tau:=\sigma_1\cap\sigma_2\in\Sigma(n-1)$. For each lift $\sigma_2^{(\beta)}$ of $\sigma_2$, let $\sigma_1^{(\alpha')}$ be the unique lift of $\sigma_1$ such that $\sigma_1^{(\alpha')}\cap\sigma_2^{(\beta)}\in\Sigma_L(n-1)$. If $\sigma_1^{(\alpha)}\cap\sigma_2^{(\beta)}\subset S'$ then $\alpha\neq\alpha'$ and
        $$z^{m(\sigma_1^{(\alpha)})-m(\sigma_1^{(\alpha')})}=z^{m(\sigma_1^{(\alpha)})-m(\sigma_2^{(\beta)})}z^{m(\sigma_2^{(\beta)})-m(\sigma_1^{(\alpha')})}$$
        is a regular function since $z^{m(\sigma_2^{(\beta)})-m(\sigma_1^{(\alpha')})}$ is nowhere vanishing on $U(\tau)$. As before $(n-1)$-separability implies the monomial $z^{m(\sigma_1^{(\alpha)})-m(\sigma_1^{(\alpha')})}$ vanishes completely on $V(\tau)$. Hence the transition map of $\cu{E}(\bb{L},{\bf{g}})|_{X_{\Sigma}^{(1)}}$ on $V(\tau)$ is given by
        $$G_{\sigma_1\sigma_2}^{sf}:=G_{\sigma_1\sigma_2}|_{V(\tau)}:1(\sigma_1^{(\alpha)})\mapsto g_{\sigma_1^{(\alpha)}\sigma_2^{(\beta)}}z^{m(\sigma_1^{(\alpha)})-m(\sigma_2^{(\beta)})}1(\sigma_2^{(\beta)}),$$
        where $\beta$ is determined by $\alpha$ as before. Then with respect to the cover $\{V_{\sigma'}\}_{\sigma'\in\Sigma_L(n)}$ of $L\backslash L^{(n-2)}$, $\{g_{\sigma_1^{(\alpha)}\sigma_2^{(\beta)}}\}$ gives a $\bb{C}^{\times}$-local system $\cu{L}$ on $L\backslash L^{(n-2)}$. For $\sigma_1\cap\sigma_2\not\subset S$, we define
        $$\Theta_{\sigma_1\sigma_2}:=(G_{\sigma_1\sigma_2}^{sf})^{-1}\circ G_{\sigma_1\sigma_2}.$$
        The diagonal entries of $\Theta_{\sigma_1\sigma_2}$ are all equal to 1 and $(n-1)$-separability implies $\Theta_{\sigma_1\sigma_2}-Id$ is nilpotent (see Lemma \ref{lem:N_commute}). This allows us to define
        $$N_{\sigma_1\sigma_2}:=\log(\Theta_{\sigma_1\sigma_2})=\log(Id+(\Theta_{\sigma_1\sigma_2}-Id))=\sum_{k=1}^{\infty}(-1)^{k-1}\frac{(\Theta_{\sigma_1\sigma_2}-Id)^k}{k}.$$
        With respect to the frame $\{1(\sigma_1^{(\alpha)})\}$, the $(\alpha,\beta)$-entry of $N_{\sigma_1\sigma_2}$ is given by
        $$N_{\sigma_1\sigma_2}^{(\alpha\beta)}=\begin{cases}
        n_{\sigma_1\sigma_2}^{(\alpha\beta)}z^{m(\sigma_1^{(\alpha)})-m(\sigma_1^{(\beta)})} & \text{ if }\alpha\neq\beta\text{ and }m(\sigma_1^{(\alpha)})-m(\sigma_1^{(\beta)})\in(\sigma_1\cap\sigma_2)^{\vee}\cap M,\\
        0 & \text{ otherwise},
        \end{cases}$$
        which can be decomposed as
        $$N_{\sigma_1\sigma_2}=\sum_{\omega'\in S_{\sigma_1\sigma_2}'}N_{\sigma_1\sigma_2}(\omega').$$
        The collection $\{\Theta_{\sigma_1\sigma_2}\}$ is obviously consistent so that $\cu{E}(\bb{L},\cu{L},\Theta)=\cu{E}(\bb{L},{\bf{g}})$.
    \end{proof}
    
    \begin{example}
    We look at the 2-fold tropical Lagrangian multi-section $\bb{L}_{a,b,c}$ over the fan of $\bb{P}^2$. Here $a,b,c>0$.
    \begin{figure}[H]
		\centering
		\includegraphics[width=115mm]{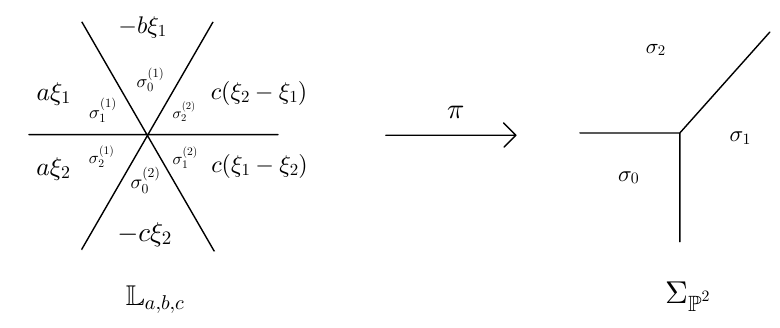}
	    \caption{The tropical Lagrangian multi-section $\bb{L}_{a,b,c}$ over $\Sigma_{\bb{P}^2}$.}
		\label{fig:L_abc_trop}
    \end{figure}
    Choose $\cu{L}$ to be the local system on $L\backslash \pi^{-1}(0)\cong\bb{R}^2\backslash\{0\}$ that has monodromy $-1$ around the minimal cone. Let $z_j^i:=Z_i/Z_j$ be the inhomogeneous coordinates on $U(\sigma_i)\cap U(\sigma_j)\subset\bb{P}^2$. The semi-flat mirror bundle $\cu{E}_0(\bb{L}_{a,b,c},\cu{L})$ on the $\bb{P}^1$-skeleton of $\bb{P}^2$ is given by the transition maps
    $$\tau_{01}^{sf}:=\begin{pmatrix}
    -\frac{1}{(z_0^1)^{a+b}} & 0\\
    0 & \frac{1}{(z_0^1)^c}
    \end{pmatrix},\,
    \tau_{12}^{sf}:=\begin{pmatrix}
    \frac{1}{(z_1^2)^a} & 0\\
    0 & -\frac{1}{(z_1^2)^{b+c}}
    \end{pmatrix},\,
    \tau_{20}^{sf}:=\begin{pmatrix}
    0 & \frac{1}{(z_2^0)^b}\\
    -\frac{1}{(z_2^0)^{a+c}} & 0
    \end{pmatrix}.$$
    We choose the wall-crossing factors to be
    $$\Theta_{01}:=\begin{pmatrix}
    1 & 0\\
    -\frac{(z_0^2)^c}{(z_0^1)^b} & 1
    \end{pmatrix},\,
    \Theta_{12}:=\begin{pmatrix}
    1 & -\frac{(z_1^0)^a}{(z_1^2)^c}\\
    0 & 1
    \end{pmatrix},\,
    \Theta_{20}:=\begin{pmatrix}
    1 & 0\\
    -\frac{(z_2^1)^b}{(z_2^0)^a} & 1
    \end{pmatrix}.$$
    One can see that the resulting toric vector bundle $\cu{E}(\bb{L}_{a,b,c},\cu{L},\Theta)$ is actually isomorphic to $E_{a,b,c}$, the toric vector bundle introduced by Kaneyama in \cite{Kaneyama_classification} using the exact sequence
    $$0\to\cu{O}_{\bb{P}^2}\to\cu{O}(aD_0)\oplus\cu{O}(bD_1)\oplus\cu{O}(cD_2)\to E_{a,b,c}\to 0.$$
    \end{example}

	\begin{remark}\label{rem:SYZ}
	From the symplectic point of view, we may think of $\Theta_{\tau}(\omega')$ as the exponentiation of the generating function of holomorphic disks emitted from the ramification locus $\omega'$, bounded by the Lagrangian multi-section and certain SYZ fibers of $p:T^*N_{\bb{R}}/M\to N_{\bb{R}}$. The exponent $m(\sigma_1^{(\alpha)})-m(\sigma_1^{(\beta)})$ in $\Theta_{\tau}(\omega')$ should be regarded as the direction of a wall if we use the polytope picture in $M_{\bb{R}}$. See \cite{Suen_TP2} for a more detailed discussion in dimension 2.
	\end{remark}

    \section{Unobstructedness in dimension 2}\label{sec:rank2}
    
    In this final section, we would like to determine when $\bb{L}$ is unobstructed when $\bb{L}$ is a combinatorially indecomposable tropical Lagrangian multi-section over a 2-dimensional complete fan. In this case, the ramification locus $S'=L^{(0)}=\pi^{-1}(0)$ is a singleton and $L\backslash\pi^{-1}(0)\cong\bb{R}^2\backslash\{0\}$ topologically. First of all, not all such tropical Lagrangian multi-sections are unobstructed.
  
    \begin{example}\label{eg:separated_obstructed}
    Consider the tropical Lagrangian multi-section $\bb{L}$ depicted as in Figure \ref{fig:separated_obstructed}.
	\begin{figure}[H]
		\centering
		\includegraphics[width=130mm]{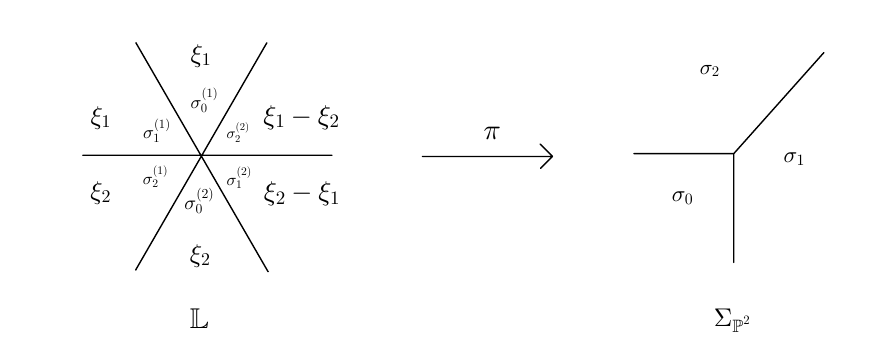}
	    \caption{}
		\label{fig:separated_obstructed}
    \end{figure}
    It is easy to see that $\bb{L}$ is maximal and separated, which implies combinatorially indecomposability by Proposition \ref{prop:com_inde_separated}. However, one checks easily that the matrices $G_{\sigma_0\sigma_1},G_{\sigma_1\sigma_2}$ are all upper-triangular while $G_{\sigma_2\sigma_0}$ must have two non-zero off-diagonal entries. Thus $\bb{L}$ must be obstructed.
    \end{example}

    Therefore, we need an extra assumption on the piecewise linear function $\varphi$ to ensure unobstructedness. We begin with two lemmas.

    \begin{lemma}\label{lem:slope_difference}
        Suppose $\bb{L}$ is a combinatorially indecomposable rank $r$ tropical Lagrangian multi-section over a complete 2-dimensional fan $\Sigma$. Let $\sigma\in\Sigma(2)$ and $\rho\subset\sigma$ be a ray. Then for $\alpha\neq\beta$, either $m(\sigma^{(\alpha)})-m(\sigma^{(\beta)})\in\rho^{\vee}\cap M$ or $m(\sigma^{(\beta)})-m(\sigma^{(\alpha)})\in\rho^{\vee}\cap M$.
    \end{lemma}
    \begin{proof}
         Since $\rho\not\subset S$, by separability, $m(\sigma^{(\alpha)})|_{\rho}\neq m(\sigma^{(\beta)})|_{\rho}$ if $\alpha\neq\beta$. In particular, $m(\sigma^{(\alpha)})-m(\sigma^{(\beta)})\neq 0$. Note that $\rho^{\vee}$ is a half plane in $M_{\bb{R}}$, we have $m(\sigma^{(\alpha)})-m(\sigma^{(\beta)})$ or $m(\sigma^{(\beta)})-m(\sigma^{(\alpha)})$ lies in $\rho^{\vee}$. Separability implies none of them can lay in $\rho^{\perp}$. Hence only one of them can lay in $\rho^{\vee}$.
    \end{proof}

    Being unobstructed also restricts the choice of the local system $\cu{L}$.
    
    \begin{lemma}\label{lem:local_system}
        If $\bb{L}$ is an unobstructed combinatorially indecomposable rank $r$ tropical Lagrangian multi-section over a complete 2-dimensional fan $\Sigma$, then $\cu{L}$ is the unique local system on $L\backslash S'$ that has monodromy $(-1)^{r+1}$ around the unique ramification point of $\pi:L\to N_{\bb{R}}$. 
    \end{lemma}
    \begin{proof}
        Since $G_{\sigma_1\sigma_2}^{sf},\dots,G_{\sigma_{k-1}\sigma_k}^{sf}$ are all diagonal, by taking the determinant of (\ref{eqn:composition}), we have
        $$(-1)^{r+1}\prod_{\alpha=1}^rg_{\sigma_k^{(\alpha)}\sigma_1^{(\alpha+1)}}^{sf}\prod_{i=1}^{k-1}\prod_{\alpha=1}^rg_{\sigma_i^{(\alpha)}\sigma_{i+1}^{(\alpha)}}^{sf}=1.$$
        Hence the monodromy of $\cu{L}$, which is given by the cyclic product of all $g_{\sigma_i^{(\alpha)}\sigma_{i+1}^{(\beta)}}$'s, is equal to $(-1)^{r+1}$. As we are in dimension 2, the monodromy around the ramification point uniquely determines the local system.
    \end{proof}
    
    \begin{remark}
    When $r=2$, the choice of the local system $\cu{L}$ has appeared in the construction of the semi-flat bundle in \cite{Fukaya_asymptotic_analysis}, Section 6.1. Fukaya pointed out in \cite{Fukaya_asymptotic_analysis}, Remark 6.4 that there should be a Floer theoretic explanation of this local system based on the orientation problem of holomorphic disks. Believing the monomial term $n_{\tau\sigma}^{(\alpha\beta)}z^{m(\sigma^{(\alpha)})-m(\sigma^{(\beta)})}$ corresponds to holomorphic disks, our calculation in Lemma \ref{lem:local_system} suggests that the present of $\cu{L}$ is due to the fact that a holomorphic disk only propagates in only \emph{one} direction; $m(\sigma^{(\alpha)})-m(\sigma^{(\beta)})$ or $m(\sigma^{(\beta)})-m(\sigma^{(\alpha)})$ but not both. The number $n_{\tau\sigma}^{(\alpha\beta)}$ will then be the weighted count of holomorphic disks (with extra boundary deformations if necessary, see Remark \ref{rmk:deformation}).
    \end{remark}
    
    Therefore, to obtain unobstructedness, it is necessary for us to choose $\cu{L}$ to be the unique local system on $L\backslash S'$ that has monodromy $(-1)^{r+1}$. In particular, by Proposition \ref{prop:independent}, we may choose the transition maps of $\cu{L}$ to be $g_{\sigma_i^{(\alpha)}\sigma_{i+1}^{(\alpha)}}^{sf}=1$, for all $i<k$, $\alpha=1,\dots,r$ and $g_{\sigma_k^{(r)}\sigma_1^{(1)}}^{sf}=(-1)^{r+1},g_{\sigma_k^{(\alpha)}\sigma_1^{(\alpha+1)}}^{sf}=1$ for $\alpha<r$. We put
    $$g_{\sigma_k\sigma_1}^{sf}:=\begin{pmatrix}
    0 & \cdots & 0 & (-1)^{r+1}\\
    1 & \cdots & 0 & 0\\
    \vdots & \ddots & \vdots & \vdots\\
    0 & \cdots & 1 & 0
    \end{pmatrix},$$
    which is the monodromy of the rank $r$ local system $\pi_*\cu{L}$ on  $N_{\bb{R}}\backslash\{0\}$. The consistency condition then becomes
    $$\theta_{\sigma_k\sigma_1}\circ\theta_{\sigma_{k-1}\sigma_k}\circ\cdots\circ\theta_{\sigma_1\sigma_2}=g_{\sigma_1\sigma_k}^{sf},$$
    where $\theta_{\sigma_i\sigma_{i+1}}$ is obtained by deleting the monomial part of $\Theta_{\sigma_i\sigma_{i+1}}$. Recall that $\Theta_{\sigma_i\sigma_{i+1}}$ is of the form $Id+N_{\sigma_i\sigma_{i+1}}$, we may write the above equation as
    \begin{equation}\label{eqn:solve_n}
        \prod_{i=1}^k(Id+n_{\sigma_i\sigma_{i+1}})=g_{\sigma_1\sigma_k}^{sf},
    \end{equation}
    Thus unobstructedness of $\bb{L}$ is equivalent to solving $n_{\sigma_i\sigma_{i+1}}$'s subordinated to the conditions
    \begin{enumerate}
        \item [(N1)] $n_{\sigma_i\sigma_{i+1}}^{(\alpha\alpha)}=0$.
        \item [(N2)] $n_{\sigma_i\sigma_{i+1}}^{(\alpha\beta)}\neq 0$ only if $m(\sigma_i^{(\alpha)})-m(\sigma_i^{(\beta)})\in (\sigma_i\cap\sigma_{i+1})^{\vee}\cap M$.
    \end{enumerate}
    Note that (N2) gives a combinatorial constraint on $\varphi$ for solving (\ref{eqn:solve_n}) as expected by Example \ref{eg:separated_obstructed}. Although Equation (\ref{eqn:solve_n}) is not easy to solve for general $r$, it has the following interesting consequence.
    
    \begin{theorem}\label{thm:dim_r}
    Let $\bb{L}$ be combinatorially indecomposable rank $r$ tropical Lagrangian multi-section over a complete 2-dimensional fan $\Sigma$. Then $$\dim_{\bb{C}}(\cu{K}(\bb{L}))\leq\frac{r(r-1)}{2}\cdot\#\Sigma(1),$$
    where $\cu{K}(\bb{L})$ is the moduli space of toric vector bundles with equivariant Chern classes determined by $\bb{L}$.
    \end{theorem}
    \begin{proof}
        The number of $n_{\sigma_i\sigma_{i+1}}$'s is exactly the number of rays in $\Sigma$ and each $n_{\sigma_i\sigma_{i+1}}$ has at most $\frac{r(r-1)}{2}$ free variables. By Theorem \ref{thm:bundle_unobstructed}, our construction extracts all the possible Kaneyama data up to equivalence. The inequality follows.
    \end{proof}
    
    \begin{remark}\label{rmk:deformation}
    The moduli space $\cu{K}(\bb{L})$ is parametrized, up to the equivalence defined in Definition \ref{def:compatible}, by the variables $n_{\sigma_i\sigma_{i+1}}^{(\alpha\beta)}$, which only depend on $N_{\sigma_i\sigma_{i+1}}$ or $\Theta_{\sigma_i\sigma_{i+1}}$. As was discussed in Remark \ref{rem:SYZ}, these parameters are related to holomorphic disks bounded by a Lagrangian multi-section and some SYZ fibers. One should expect that these variables are actually mirror to the moduli parameters of $A_{\infty}$-deformations of the Lagrangian multi-section.
    \end{remark}

    Finally, we give an explicit description of the combinatorial obstruction for solving (\ref{eqn:solve_n}) in the case $r=2$. This condition is particularly easy to check. Let's recall Lemma \ref{lem:slope_difference}. In the rank 2 case, it means for any $\sigma_1,\sigma_2\in\Sigma(2)$ that intersect along an edge, we are always allowed to put 3 non-zero entries in the $2\times 2$ matrices $G_{\sigma_1\sigma_2}$. Without loss of generality, we may arrange $\sigma_1^{(1)},\sigma_2^{(1)},\dots,\sigma_k^{(1)},\sigma_1^{(2)},\sigma_2^{(2)},\dots,\sigma_k^{(2)}$ in an anti-clockwise manner such that the matrix $G_{\sigma_k\sigma_1}$ is of form
    $$\begin{pmatrix}
    z^{m(\sigma_k^{(1)})-m(\sigma_1^{(1)})} & -z^{m(\sigma_k^{(2)})-m(\sigma_1^{(1)})}\\
    z^{m(\sigma_k^{(1)})-m(\sigma_1^{(2)})} & 0
    \end{pmatrix}$$
    and all the remaining $G_{\sigma_i\sigma_{i+1}}$ are either upper-triangular or lower-triangular.
    
    \begin{definition}
    Let $\bb{L}$ be a tropical Lagrangian multi-section over a complete fan $\Sigma$. The \emph{slope matrix} $M_{\sigma_1\sigma_2}$ associated to $\sigma_1,\sigma_2\in\Sigma(n)$ is the matrix given by
    $$M_{\sigma_1\sigma_2}^{(\alpha\beta)}:=\begin{cases}
    m(\sigma_1^{(\alpha)})-m(\sigma_2^{(\beta)}) & \text{ if }m(\sigma_1^{(\alpha)})-m(\sigma_2^{(\beta)})\in(\sigma_1\cap\sigma_2)^{\vee}\cap M,\\
    \infty & \text{ otherwise}.
    \end{cases}$$
    One associates to the slope matrix $M_{\sigma_1\sigma_2}$ the \emph{monomial matrix}
    $$Z_{\sigma_1\sigma_2}^{(\alpha\beta)}:=\begin{cases}
    z^{m(\sigma_1^{(\alpha)})-m(\sigma_2^{(\beta)})} & \text{ if }m(\sigma_1^{(\alpha)})-m(\sigma_2^{(\beta)})\in(\sigma_1\cap\sigma_2)^{\vee}\cap M,\\
    0 & \text{ otherwise}.
    \end{cases}$$
    We call a slope matrix \emph{upper-triangular} (resp. \emph{lower-triangular}) if the associated monomial matrix is upper-triangular (resp. lower-triangular).    
    \end{definition}

    By Lemma \ref{lem:local_system}, the local system $\cu{L}$ needs to be chosen to have monodromy $-1$ around the ramification point. With the above choice of arrangement convention, it is necessary that the coefficient matrix of the composition $G_{\sigma_{k-1}\sigma_k}\circ\cdots\circ G_{\sigma_1\sigma_2}$ takes the form
    \begin{equation}\label{eqn:inverse}
    \begin{pmatrix}
    0 & -1\\
    1 & 1
    \end{pmatrix}.
    \end{equation}

    \begin{definition}\label{def:slope_condition}
    A combinatorially indecomposable rank $2$ tropical Lagrangian multi-section $\bb{L}$ over a complete 2-dimensional fan $\Sigma$ is said to be satisfying the \emph{slope condition} if under the above arrangement convention, one of the following conditions is satisfied.
    \begin{enumerate}
        \item If $M_{\sigma_{k-1}\sigma_k}$ is upper-triangular, there is at least one $i<k-1$ such that $M_{\sigma_i\sigma_{i+1}}$ is lower-triangular.
        \item If $M_{\sigma_{k-1}\sigma_k}$ is lower-triangular, there exists some $i,j$ with $1\leq i<j<k-1$, such that $M_{\sigma_j\sigma_{j+1}}$ is upper-triangular and $M_{\sigma_i\sigma_{i+1}}$ is lower-triangular.
    \end{enumerate}
    \end{definition}
    
    \begin{theorem}\label{thm:rank2_unobstructed}
    A combinatorially indecomposable rank $2$ tropical Lagrangian multi-section $\bb{L}$ over a 2-dimensional complete fan $\Sigma$ is unobstructed if and only if it satisfies the slope condition.
    \end{theorem}
    \begin{proof}
        If $\bb{L}$ is unobstructed and $G_{\sigma_{k-1}\sigma_k}$ is of upper-triangular type, then it is clear that we need a lower-triangular type matrix to bring it into the required form (\ref{eqn:inverse}). Suppose $G_{\sigma_{k-1}\sigma_k}$ is of lower-triangular type. There must be some $j<k-1$ so that $G_{\sigma_j\sigma_{j+1}}$ is of upper-triangular type. If there are no $i<j$ for which $G_{\sigma_i\sigma_{i+1}}$ is of lower triangular type, the composition $G_{\sigma_{k-1}\sigma_k}\circ\cdots\circ G_{\sigma_1\sigma_2}$ will then take the form
        $$\begin{pmatrix}
        1 & 0\\
        * & 1
        \end{pmatrix}
        \begin{pmatrix}
        1 & *\\
        0 & 1
        \end{pmatrix}=
        \begin{pmatrix}
        1 & *\\
        * & *
        \end{pmatrix},$$
        which can never have the required form (\ref{eqn:inverse}). It remains to prove the converse. In the upper-triangular case, let $i<k-1$ be the first index for which $M_{\sigma_i\sigma_{i+1}}$ is lower-triangular. Then  
        $$(G_{\sigma_{k-1}\sigma_k}\circ G_{\sigma_{k-1}\sigma_{k-2}}\circ\cdots\circ G_{\sigma_{i+1}\sigma_{i+2}})\circ G_{\sigma_i\sigma_{i+1}}=\begin{pmatrix}
        1 & a\\
        0 & 1
        \end{pmatrix}
        \begin{pmatrix}
        1 & 0\\
        b & 1
        \end{pmatrix}=
        \begin{pmatrix}
        1+ab & a\\
        b & 1
        \end{pmatrix},$$
        by choosing $a=-1,b=1$, we obtain (\ref{eqn:inverse}). Then we simply choose the remaining matrices to be the identity to obtain $G_{\sigma_{k-1}\sigma_k}\circ\cdots\circ G_{\sigma_1\sigma_2}=G_{\sigma_k\sigma_1}^{-1}$. For the lower-triangular case, let $i<j<k-1$ be the first index for which $M_{\sigma_j\sigma_{j+1}}$ is upper-triangular and $M_{\sigma_i\sigma_{i+1}}$ is lower-triangular. Then we have
        $$G_{\sigma_{k-1}\sigma_k}\circ\cdots\circ G_{\sigma_j\sigma_{j+1}}\circ\cdots\circ G_{\sigma_i\sigma_{i+1}}=\begin{pmatrix}
        1 & 0\\
        a & 1
        \end{pmatrix}
        \begin{pmatrix}
        1 & b\\
        0 & 1
        \end{pmatrix}
        \begin{pmatrix}
        1 & 0\\
        c & 1
        \end{pmatrix}=
        \begin{pmatrix}
        1+bc & b\\
        a+c+abc & 1
        \end{pmatrix}.$$
        Choose $b=-1,c=1$ and $a$ be arbitrary. Then the triple product equals (\ref{eqn:inverse}). Again, by choosing the remaining matrices to be the identity, we obtain $G_{\sigma_{k-1}\sigma_k}\circ\cdots\circ G_{\sigma_1\sigma_2}=G_{\sigma_k\sigma_1}^{-1}$.
    \end{proof}

    The proof of Theorem \ref{thm:rank2_unobstructed} also sharpens the inequality in Theorem \ref{thm:dim_r}.
    
    \begin{corollary}\label{cor:dim}
    If $\bb{L}$ is a combinatorially indecomposable rank $2$ tropical Lagrangian multi-section over a complete 2-dimensional fan $\Sigma$, then we have $\dim_{\bb{C}}(\cu{K}(\bb{L}))\leq\#\Sigma(1)-1$.
    \end{corollary}
    \begin{proof}
        In the proof of Theorem \ref{thm:rank2_unobstructed}, the equation $1+ab=0$ in the upper-triangular case or $1+bc=0$ in the lower-triangular case cut down the dimension by $1$. By Theorem \ref{thm:bundle_unobstructed}, our construction extracts all the possible toric structures with fixed equivariant Chern class, which is determined by $\bb{L}$. Hence $\dim_{\bb{C}}(\cu{K}(\bb{L}))\leq\#\Sigma(1)-1$.
    \end{proof}

	\bibliographystyle{amsplain}
	\bibliography{geometry}

\end{document}